\documentclass[a4paper, 11pt, reqno]{amsart}

\usepackage[utf8]{inputenc}
\usepackage[T1]{fontenc}
\usepackage{mathtools}
\usepackage{lmodern}
\usepackage[english]{babel}
\usepackage{hyphenat}
\usepackage{booktabs}
\usepackage{float}
\usepackage{enumitem}
\usepackage{tikz,pgf}
\usetikzlibrary{calc}
\usetikzlibrary{through}
\usepackage{tkz-euclide}
\usepackage{palatino}
\usepackage{graphicx}
\usepackage{xcolor}
\usepackage{hyperref}
\hypersetup{
    colorlinks,
    linkcolor={red!50!black},
    citecolor={blue!50!black},
    urlcolor={blue!80!black}
}
\usepackage[nameinlink]{cleveref}

\theoremstyle{plain}
\newtheorem{lemma}{Lemma}[section]
\newtheorem{prop}[lemma]{Proposition}
\newtheorem{theorem}[lemma]{Theorem}
\newtheorem{corollary}[lemma]{Corollary}
\theoremstyle{definition}
\newtheorem{definition}[lemma]{Definition}
\newtheorem{es}[lemma]{Example}
\newtheorem{rmk}[lemma]{Remark}

\tikzstyle{point}=[circle, draw, fill=black, inner sep=0pt, minimum size=4pt]
\tikzstyle{line}=[line width=1.5pt, black!70!white]

\newcommand{\R}{\mathbb{R}}
\newcommand{\C}{\mathbb{C}}
\newcommand{\p}{\mathbb{P}}
\newcommand{\sL}{\mathcal{L}}
\newcommand{\sE}{\mathcal{E}}
\newcommand{\sW}{\mathcal{W}}
\newcommand{\sU}{\mathcal{U}}
\newcommand{\sV}{\mathcal{V}}
\newcommand{\sF}{\mathcal{F}}
\newcommand{\de}{\partial}

\newcommand{\nb}[2]{\textsl{{NB}.{#1}.{#2}}}

\newcommand{\iii}{\textit{i}\,}
\newcommand{\rk}{\ensuremath{\mathrm{rk}}}

\newcommand{\iso}{\mathcal{Q}_{\mathrm{iso}}}

\newcommand{\SO}{\operatorname{SO}}
\newcommand{\Eig}[1]{\mathcal{E}\!\left( {#1} \right)}

\newcommand{\scl}[2]{\left\langle {#1}, {#2} \right\rangle}

\newcommand\scalemath[2]{\scalebox{#1}{\mbox{\ensuremath{\displaystyle #2}}}}

\title{Eigenpoint collinearities of plane cubics}
\author[Valentina Beorchia]{Valentina Beorchia$^{\circ}$}
\address[\textsc{Valentina Beorchia}]{University of Trieste,
Department of Mathematics, Informatics and Geosciences,
Via Valerio 12/1, 34127 Trieste, Italy}
\email{beorchia@units.it}
\thanks{$^{\circ}$The researcher is a member of ``Gruppo Nazionale per le Strutture Algebriche, Geometriche e le loro Applicazioni'', INdAM. She is partially supported by MUR funds: PRIN project GEOMETRY OF ALGEBRAIC STRUCTURES: MODULI, INVARIANTS, DEFORMATIONS, PI Ugo Bruzzo, Project code: 2022BTA242.}
\author[Matteo Gallet]{Matteo Gallet$^{\diamond}$}
\address[\textsc{Matteo Gallet}]{University of Trieste,
Department of Mathematics, Informatics and Geosciences,
Via Valerio 12/1, 34127 Trieste, Italy}
\email{matteo.gallet@units.it}
\thanks{$^{\diamond}$The researcher is a member of ``Gruppo Nazionale per le Strutture Algebriche, Geometriche e le loro Applicazioni'', INdAM}
\author[Alessandro Logar]{Alessandro Logar}
\address[\textsc{Alessandro Logar}]{University of Trieste,
Department of Mathematics, Informatics and Geosciences,
Via Valerio 12/1, 34127 Trieste, Italy}
\email{logar@units.it}
\date{}

\begin{document}

\begin{abstract}
 Given a ternary homogeneous polynomial, the fixed points of the map from~$\p^2$ to itself defined by its gradient are called its \emph{eigenpoints}. We focus on cubic polynomials, and analyze configurations of eigenpoints that admit one or more alignments. We give a classification and explicit equations, in the coordinates of the points, of all configurations: this is accomplished by using both geometric techniques and by an extensive use of computer algebra.
\end{abstract}

\maketitle

\section{Introduction}
\label{introduction}

Tensors are natural generalizations of matrices in higher dimension.
Important notions related to matrices, as for example those of \emph{rank} and of \emph{eigenvector},
can be generalized to tensors and may provide information about them.
Since roughly twenty years ago, eigenvectors of tensors have been object of study from various points of view.
Our interest will be focused on a specific class of symmetric tensors, namely cubic ternary forms, and particular geometric configurations of their eigenvectors.
Before delving into the more technical aspects of our problem,
we begin with a short resume of the existing literature on eigenvectors of tensors.

\subsection*{Eigenvectors of tensors in the literature}
In 2005, Lim \cite{Lim} and Qi \cite{Qi} independently introduced the notions of eigenvalues and eigenvectors for tensors.
Since then, these two concepts have been of interest in several applications,
as in the study of hypergraphs and dynamical systems;
see \cite[Section~4]{QZ} for an introduction or \cite{GMV} for recent developments.
Eigenvectors of tensors also play an important role in the best rank-one approximation problem,
which is relevant in data analysis and signal processing.
Indeed, by Lim's Variational Principle \cite{Lim}, the critical rank-one symmetric tensors for a symmetric tensor~$f$ are of the type~$v^d$, where $v$ is an eigenvector of~$f$.
This has applications in low-rank approximation of tensors (see \cite{OttSod}) as well as maximum likelihood estimation in algebraic statistics.
For another application in optimization, consider the problem of maximizing a polynomial function~$f$ over the unit sphere in~$\R^{n+1}$:
the eigenvectors of the symmetric tensor~$f$ are critical points of this optimization problem.

The geometry of eigenvectors is intimately related with the Waring decomposition of a polynomial,
corresponding to the symmetric Tucker decomposition of the associated symmetric tensor,
as clarified in~\cite{DOT} and in~\cite{Ott}.
Indeed, any best rank~$k$ approximation of a symmetric tensor, when it exists, lies in the linear space,
called \emph{critical space}, spanned by the rank~$1$ tensors of the type $v_i^{\otimes d}$, where $v_i$ varies among all the eigenvectors.
Therefore, a deep comprehension of the geometry of eigenvectors will
go towards an improvement on the insight of low rank approximation problems.
A first evidence in this direction is given by the so called ODECO (``orthogonally decomposable'') tensors,
that is, symmetric tensors that admit a decomposition
$\sum _{i=0}^n v_i ^{\otimes d}$ with $v_0, \dotsc, v_n$ an orthogonal family (see \cite{Rob, BDHE} for further details);
the $v_i$'s turn out to be eigenvectors.
Finally, in \cite{OO} Oeding and Ottaviani employ eigenvectors of tensors in an algorithm to compute Waring decompositions of homogeneous polynomials.
Ottaviani points out in \cite[Section~8]{Ottaviani24} the close connection between eigenvectors of specific symmetric tensors and the so-called L\"uroth quartics.

\subsection*{Eigenschemes of homogeneous ternary forms}
Given a vector space $V$, we fix an isomorphism $V \cong V^\vee$ where $V^\vee$ is the dual vector space.
A non-zero vector $v \in V$ is an {\it eigenvector} of a tensor
 $T\in V ^{\otimes d} \cong ( V^\vee)^{\otimes d-1}\otimes V$ if there exists $\lambda\in \C$ such that
 $T(v^{\otimes d-1})= \lambda v$.
In the present paper, we will focus on the geometry of configurations of eigenvectors of order $3$ ternary symmetric tensors,
that is elements of ${\rm Sym}^3 \C^3$.

Recall that, by fixing a nondegenerate quadratic form and an orthonormal basis of $\C^3$, the space ${\rm Sym}^d \C^3$ can be identified with the space of homogeneous ternary forms
$\C[x,y,z]_d$, and it turns out that the general definition of eigenvector $v$ for $f \in \C[x,y,z]_d$ specializes
to $\nabla f (v)=\lambda v$, see for instance \cite[Section 1]{ASS}.
Such a condition is preserved under scalar multiplication,
so it is natural to regard eigenvectors as points in~$\p^2$, and to talk about \emph{eigenpoints}. Moreover,
by construction, the set of eigenpoints of $f$ is the zero locus of the ideal of $2 \times 2$ minors of the matrix
\begin{equation}
\label{eq:def_matrix}
\begin{pmatrix}
    x & y & z \\
    \partial_x f  & \partial_y f & \partial_z f
\end{pmatrix} \,,
\end{equation}
thus it can be given a natural structure of determinantal scheme called \emph{eigenscheme}, which we denote by $\Eig{f}$.
Similarly, a natural scheme structure can be given to the eigenpoints of any
tensor $T \in (\C^3)^{\otimes d}$, see \cite[Section 1]{ASS}, and such a scheme will be denoted by $\Eig{T}$.

It has been proved in \cite[Corollary 5.8]{Abo} that $\Eig{f}$ is in general $0$-dimensional and reduced,
and its degree has been determined in \cite[Theorem 2.1]{CartSturm}.
In the particular case of $d=3$ we generally have
$\deg \Eig{f}=7$; however,
 there are cases when it is one-dimensional, or non-reduced.

A first geometric characterization of eigenschemes of ternary tensors has been given in
\cite[Theorem 5.1]{ASS}, and it states that seven points in~$\p^2$ are eigenpoints of some tensor if and only if no six of them lie on a conic.
Moreover, by \cite[Theorem 5.7]{BGV},
the eigenscheme~$\Eig{f}$ of a general ternary cubic contains no triple of collinear points.

However, there are several examples of cubic forms, whose eigenscheme contains one or more triples of aligned points,
as for instance the Fermat cubic polynomial $f=x^3+y^3+z^3$.
Understanding these situations requires a careful analysis of the alignement conditions.

\subsection*{Our results}
The goal of this paper is to classify all the situations when we have one or more triples of aligned points inside a $0$-dimensional reduced eigenscheme of a ternary cubic form.
More precisely, we
\begin{itemize}
    \item determine conditions imposed on the eigenpoints and on ternary cubic forms by the presence of aligned eigenpoints;
    \item compute the degree of the closure of the locus in~$\p^9$ of ternary cubic forms that admit an aligned triple of eigenpoints;
    \item characterize the situations yielding a one-dimensional eigenscheme;
    \item classify all possible combinatorial configurations of aligned eigenpoints when the eigenscheme is $0$-dimensional and reduced.
\end{itemize}
All the conditions on the eigenpoints will be expressed in terms of their homogeneous coordinates, and the equations that we will find have to be though as equations in $(\p^2)^7=\p^2 \times \p^2 \times \dotsb \times \p^2$.

Finally, we would like to point out that our approach gives some hints towards the problem of finding equations for the locus of $7$-uple of points, which form an eigenscheme of some symmetric tensor; this is an open question, and it has been posed in \cite{ASS} and \cite{Ottaviani24}.

\subsection*{Content of the paper}
Here is a description of the content of each section.
\Cref{invariance} describes the action of special orthogonal matrices on ternary forms, which we repeatedly use in our analyses to "pin down" one of the eigenpoints to two specific cases, thus greatly speeding up the symbolic computations we rely on.
\Cref{conditions} begins our analysis by examining what conditions on the space of ternary cubics are imposed by the existence of aligned eigenpoints. In particular, three kinds of conditions, named~$\delta_1$, $\bar{\delta}_1$, and~$\delta_2$, emerge and play a prominent role in describing configurations of two aligned triples that share a point, here called $V$- configurations.
\Cref{V-configurations} sets the spotlight on $V$- configurations, exhausting all the possible special cases that may appear.
\Cref{locus_one_alignment} introduces the locus of ternary cubics with an alignment in the eigenscheme, proves its irreducibility, and computes its degree.
\Cref{positive_dim} classifies positive-dimensional eigenschemes of ternary cubics.
\Cref{further_alignments} examines all the possible configurations of alignments in a $0$-dimensional reduced eigenscheme.

\subsection*{Symbolic computation}
Several proofs employ computer algebra verifications in an essential way.
We rely on SageMath \cite{SageMath} for the operations we need
(mainly, manipulations of polynomials and polynomial ideals).
To allow the reader for an easier use of these computations, we have prepared Jupyter notebooks, available at \cite{Notebooks}.
Each notebook has a name of the form \nb{XX}{CODE} where \textit{XX} denotes the section and \textit{CODE} refers to the result it proves.

\subsection*{Notation and preliminaries}
We work in $\p^2_\C$, endowed with the canonical projective coordinate system.
We denote by~$P \vee Q$ the line through two distinct points~$P$ and~$Q$.
We denote by~$\iso$ the so-called \emph{isotropic conic}, namely the conic of equation
\[
 x^2 + y^2 + z^2 = 0 \,.
\]
Since there will be no ambiguity, sometimes we will identify $\iso$ with the
defining polynomial.

The quadratic form associated to~$\iso$ is denoted $\left\langle \cdot, \cdot \right\rangle$ and throughout the paper, the notion of \emph{orthogonality} refers to this quadratic form.

For $P_i=(A_i:B_i:C_i)$ and $P_j=(A_j:B_j:C_j)$, we shall write
\[
  s_{ij} = \scl{P_i}{P_j} = A_i A_j + B_i B_j + C_i C_j \,.
\]
Such an expression has to be intended as a bihomogeneous polynomial in the coordinates of~$P_i$ and~$P_j$. If we think of the coordinates as indeterminates, the zero set of~$s_{ij}$ is well defined in $\p^2 \times \p^2$.

Similarly, by $P_1 \times P_2$ we shall denote the \emph{cross product} of the representing vectors
\begin{equation}
\label{eq:cross_product}
  P_1 \times P_2 =
  (B_1 C_2 - C_1 B_2, \, C_1 A_2 - A_1 C_2, \, A_1 B_2 - B_1 A_2) \,.
\end{equation}
The resulting projective point is well defined. Such a notation will be used, for instance, in \Cref{definition:delta1}.

In a similar fashion, whenever we define a point in~$\p^2$ by means of the expressions we introduce in the text, as for example in \Cref{lemma_delta_case2}, where we write
\[
  P_3 = (s_{12}^2+s_{11}s_{22}) \, P_1 - 2s_{11}s_{12} \, P_2 \,,
\]
what we mean is that the right hand side is a (multi-)homogeneous polynomial, which therefore defines a rational function
\[
  \p^2 \times \p^2 \dashrightarrow \p^2
\]
and the point on the left hand side is the image of $P_1$ and $P_2$ under this map.

As said above, given a homogeneous form $f \in \C[x,y,z]_d$ of degree~$d$, the \emph{eigenscheme}~$\Eig{f} \subset \p^2$ of~$f$ is the determinantal scheme defined by the $2 \times 2$ minors of the matrix in \Cref{eq:def_matrix}.
Throughout the paper, we denote them by $g_1, g_2$ and $g_3$ and we fix the following choice:
\begin{equation}
\label{eq:def_minors}
 g_1 = x \partial_y f - y \partial_x f \,, \quad
 g_2 = x \partial_z f - z \partial_x f \,, \quad
 g_3 = y \partial_z f - z \partial_y f \,.
\end{equation}

The following result (see \cite[Theorem 2.1]{CartSturm}, \cite{ASS}, \cite{OO}, and \cite[Equation~5.2]{Abo}) determines the degree of zero-dimensional eigenschemes.

\begin{theorem}
\label{theorem:nonempty}
Let $d \ge 2$ and let $T \in (\C^{n+1})^{\otimes d}$.
If $\dim \Eig{T}=0$, then
\[
  \deg \Eig{T} =
  \frac{(d-1)^{n+1}-1}{d-2} =
  \sum_{i=0}^{n} (d-1)^i \,.
\]
In the matrix case, that is $d = 2$, we use the formula on the right, which evaluates to $n+1$.
\end{theorem}

We point out that \Cref{theorem:nonempty} holds also in the non-reduced case. However, a result by H.\ Abo (\cite[Corollary 5.8]{Abo}) guarantees that general forms have $0$-dimensional and reduced eigenschemes.

\begin{definition}
\label{definition:eigendiscriminant}
An order~$3$ tensor $T \in (\C^3)^{\otimes 3}$ is called \emph{regular} if its eigenscheme is reduced of dimension~$0$.
We set
\[
 \Delta_{3,3} := \{[T]\in \p \bigl( (\C^3)^{\otimes 3} \bigr) \ | \ \Eig{T} \textrm{\ is \ not \ regular} \} \subset \p \bigl( (\C^3)^{\otimes 3} \bigr) \,.
\]
The locus~$\Delta_{3,3}$ is called the \emph{eigendiscriminant} (see \cite[Definition 5.5]{Abo}) and is an irreducible hypersurface (see \cite[Corollary 5.8]{Abo}).
\end{definition}

As already mentioned, a regular eigenscheme of a cubic form never contains six points on a conic by \cite[Theorem 5.1]{ASS}. As a consequence we have the following result.

\begin{lemma}
\label{lemma:no_4_aligned}
Given $f \in \C[x,y,z]_3$,
if $\Eig{f}$ is regular, then it contains no $4$ or more aligned points.
Moreover, if $\Eig{f}$ contains two aligned triples, they must share a point.
\end{lemma}

It can be expected that the presence of one or more aligned triples of eigenpoints should impose particular geometric constraints, since by \cite[Theorem 5.7]{BGV} this is not the general case.

We conclude this section by recalling the notion of \emph{Geiser map} and by explaining the connection between its contracted locus with alignments of eigenpoints; such a point of view  will be useful in \Cref{V-configurations}.

\begin{definition}
The Geiser map associated with a zero dimensional eigenscheme $\Eig{f}$ is the rational map defined by
\[
  \gamma_{\Eig{f}} \colon \p ^2 \dasharrow \p^2, \quad
  \gamma_{\Eig{F}} (P) = \bigl( g_3(P):-g_2(P):g_1(P) \bigr) \,,
\]
where $g_1, g_2, g_3$ are as in \Cref{eq:def_minors}.
\end{definition}

Geiser maps are a classical topic and several of their properties are understood.
As an example, the map~$\gamma_{\Eig{f}}$ is generically finite of degree~$2$, see for instance \cite[Section~8.7.2]{Dolgachev}.

\begin{lemma}
The Geiser map~$\gamma_{\Eig{f}}$ is surjective and its fiber over a point $Q = (a:b:c)$ is given by
\begin{equation}
\label{eq:fibers}
  \left\{
  \begin{array}{l}
    a x + by + cz = 0 \,, \\[2pt]
    a \, \partial_x f + b \, \partial_y f + c \, \partial_z f = 0 \,,\\
  \end{array}
  \right.
\end{equation}
that is, the intersection between the polar line~$\ell_Q$ relative to the isotropic conic~$\iso$ and~$\mathrm{Pol}_Q f$, the
first polar curve of~$f$ with respect to~$Q$.

In particular, the only possible curves contracted by the Geiser map~$\gamma_{\Eig{f}}$ are lines.
\end{lemma}

\begin{proof}
We observe that for any point $P=(A:B:C) \in \p^2 \setminus \Eig{f}$, the homogeneous coordinates
of the image $\gamma_{\Eig{f}}(P) = \bigl( g_3(P): -g_2(P): g_1(P) \bigr)$ are the ones of the unique point in the intersection of the two lines
\[
  Ax + By+ Cz = 0 \,, \qquad
  \partial_x f(P) \, x + \partial_y f(P) \, y + \partial_z f(P) \, z = 0 \,.
\]
So for any $Q = (a:b:c) \in \p^2$, the fiber~$\gamma_{\Eig{f}}^{-1}(Q)$ consists of the points $P \in \p^2$ such that
\begin{equation}
\label{eq:polars}
  Aa + Bb+ Cc = 0, \quad
  \partial_x f(P) \, a + \partial_y f(P) \, b + \partial_z f(P) \, c = 0 \,,
\end{equation}
which proves that $\gamma_{\Eig{f}}$ is surjective and that \Cref{eq:fibers} holds.
\end{proof}

\begin{prop}
\label{proposition:contract_aligned}
Suppose that $\Eig{f}$ is a zero dimensional eigenscheme of a cubic ternary form.
If $\Eig{f}$ contains a triple of points on a line~$\ell \subset \p^2$, then the Geiser map~$\gamma_{\Eig{f}}$ contracts~$\ell$.
\end{prop}

\begin{proof}
We identify the codomain $\p^2$ with $\p\bigl(I_{\Eig{f}}(3)^{\vee}\bigr)$, namely the projectivization of the dual of the space of degree~$3$ forms in the homogeneous ideal of~$\Eig{f}$.
For any $P \in \p^2 \setminus \Eig{f}$, the point~$\gamma_{\Eig{f}} (P)\in \p \bigl(I_{\Eig{f}}(3)^{\vee}\bigr)$ corresponds to the line~$\ell_P$ given by
\[
  g_3 (P) \, x - g_2(P) \, y + g_1(P) \, z = 0 \,,
\]
which in turn determines the following
pencil of cubics $\ell_P$ in the linear system $\p \bigl(I_{\Eig{f}}(3)\bigr)$:
\[
  \lambda \cdot \bigl( g_2(P) \, g_3  -g_3(P) \, g_2\bigr) + \mu \cdot \bigl( g_3(P) \, g_1 - g_1(P) \, g_3 \bigr) \,,
\]
where $(\lambda : \mu) \in \p^1$, assuming that $g_1(P) \neq 0$ (the other cases are similar).
It is clear that the base locus of the pencil~$\ell_P$ contains $\Eig{f} \cup \{P\}$.
So $\gamma_{\Eig{f}} \colon \p^2 \dasharrow \p \bigl( I_{\Eig{f}}(3)^\vee \bigr)$
associates with each point $P \in \p^2 \setminus \Eig{f}$ the pencil of cubics through $\Eig{f} \cup P$.

If $\Eig{f}$ contains a triple of points on a line~$\ell$, for any $P \in \ell \setminus \Eig{f}$ such a pencil of cubics has $\ell$ as a fixed component and the other component varies in a pencil of conics through the remaining $4$ points, so $\gamma_{\Eig{f}}$ is constant on~$\ell$.
\end{proof}

\section{Invariance under the action of orthogonal matrices}
\label{invariance}

In what follows, it will be useful to fix particular coordinates for points and lines related to eigenschemes.
To do that, we employ a property of invariance of eigenschemes with respect to the action of the following group.

\begin{definition}
We define $\SO_3(\C)$ to be the complexification of the group of special orthogonal real matrices, namely
\[
  \SO_3(\C) :=
  \bigl\{
    M \in \mathrm{GL}_3(\C) \, \mid \,
    M \prescript{t} {}M = I_3 \ \text{and} \ \det(M) = 1
  \bigr\} \,.
\]
The group $\SO_3(\C)$ acts on $\C^3$ by matrix multiplication:
\[
  \begin{array}{ccc}
    \SO_3(\C) \times \C^3 & \rightarrow & \C^3 \\
    (M, v) & \mapsto & Mv
  \end{array}
\]
Since all the elements of~$\SO_3(\C)$ are invertible, the latter action descends to an action on $\p^2$.

Moreover, the group~$\SO_3(\C)$ acts also on ternary forms via
\[
  M \cdot f (x,y,z) := f(M^{-1} \cdot \prescript{t} {}( x \ y \ z ) ).
\]
\end{definition}

\begin{prop}
\label{proposition:two_orbits}
The action of~$\SO_3(\C)$ on $\p^2$ has two orbits:
\begin{align*}
  \mathcal{O}_1 &:=
  \bigl\{
    P \in \p^2 \, \mid \,
    P = (a:b:c) \ \text{with} \ a^2 + b^2 + c^2 = 0
  \bigr\} \\
  \mathcal{O}_2 &:= \p^2 \setminus \mathcal{O}_1
\end{align*}
A representative for~$\mathcal{O}_1$ is~$(1:\iii:0)$ and a
representative for~$\mathcal{O}_2$ is~ $(1:0:0)$.
The orbit~$\mathcal{O}_1$ is the set of points of the isotropic conic~$\iso$.
\end{prop}
\begin{proof}
Suppose that $P \in \p^2$ and $P = (a:b:c)$ with $a^2 + b^2 + c^2 = 0$.
We produce a matrix $M \in \SO_3(\C)$ such that $M \left(\begin{smallmatrix} 1 \\ \iii \\ 0 \end{smallmatrix}\right)$ and $\left(\begin{smallmatrix} a \\ b \\ c \end{smallmatrix}\right)$ are proportional.
Up to relabeling the coordinates, we can suppose that $a \neq 0$.
Hence, by rescaling the coordinates of~$P$, we have $P = (1: b: c)$ with $b^2 + c^2 = -1$.
One can check that the matrix
\[
  M :=
  \begin{pmatrix}
    -1 & 0 & 0 \\
    0 & \iii b & -\iii c \\
    0 & \iii c & \iii b
  \end{pmatrix}
\]
satisfies the requirements.

Now, suppose that $P \in \p^2$ and $P = (a:b:c)$ with $a^2 + b^2 + c^2 \neq 0$.
Up to rescaling, we can suppose that $a^2 + b^2 + c^2 = 1$.
Again, we produce a matrix $M \in \SO_3(\C)$ such that $M \left(\begin{smallmatrix} 1 \\ 0 \\ 0 \end{smallmatrix}\right)$ and~$\left(\begin{smallmatrix} a \\ b \\ c \end{smallmatrix}\right)$ are proportional.
First of all, suppose that $b^2 + c^2 \neq 0$ and let $\omega$ be a root of the polynomial $t^2 - (b^2 + c^2)$ in $\C[t]$.
Then, the matrix
\[
  M :=
  \begin{pmatrix}
    a & \omega & 0 \\
    b & -\frac{ab}{\omega} & \frac{c}{\omega} \\
    c & -\frac{ac}{\omega} & -\frac{b}{\omega}
  \end{pmatrix}
\]
satisfies the requirements.
With the same technique, if $a^2 + c^2 \neq 0$, we can produce a matrix $M \in \SO_3(\C)$ that maps $\left(\begin{smallmatrix} 0 \\ 1 \\ 0 \end{smallmatrix}\right)$ to $\left(\begin{smallmatrix} a \\ b \\ c \end{smallmatrix}\right)$; similarly, when $a^2 + b^2 \neq 0$, we can map $\left(\begin{smallmatrix} 0 \\ 0 \\ 1 \end{smallmatrix}\right)$ to $\left(\begin{smallmatrix} a \\ b \\ c \end{smallmatrix}\right)$.
Since $\left(\begin{smallmatrix} 1 \\ 0 \\ 0 \end{smallmatrix}\right)$, $\left(\begin{smallmatrix} 0 \\ 1 \\ 0 \end{smallmatrix}\right)$, and $\left(\begin{smallmatrix} 0 \\ 0 \\ 1 \end{smallmatrix}\right)$ are all $\mathrm{SO}_3(\C)$-equivalent, the only case to consider is when
\[
  b^2 + c^2 = a^2 + c^2 = a^2 + b^2 = 0 \,,
\]
which, however, can never occur.
\end{proof}

The following result is well known; we recall it for the sake of completeness.

\begin{prop}
Let $M \in \SO_3(\C)$ and let $f$ be a ternary cubic.
Let $P = (A: B: C)$ be a point in~$\p^2$.
Then we have
\[
  P \in \Eig{f}
  \quad \text{if and only if} \quad
  M \cdot \prescript{t} {}(A \ B \ C) \in \Eig{M \cdot f} \,.
\]
\end{prop}
\begin{proof}
In this proof, for convenience,
we consider the transpose of the defining matrix of an eigenscheme.
A point~$P = (A: B: C)$ is an eigenpoint for~$f$ if and only if
\begin{equation}
\label{eq:def_matrix_M}
  \mathrm{rk}
  \begin{pmatrix}
    A & \de_x f(P) \\
    B & \de_y f(P) \\
    C & \de_z f(P)
  \end{pmatrix}
  = 1 \,
 \quad \text{or, equivalently} \quad
  \mathrm{rk} \ M \cdot
  \begin{pmatrix}
    A & \de_x f(P) \\
    B & \de_y f(P) \\
    C & \de_z f(P)
  \end{pmatrix}
  = 1 \,.
\end{equation}
By setting $\prescript{t} {}(A' \ B' \ C' ) := M \cdot \prescript{t} {}(A \ B \ C) $ and~$Q := (A':B':C')$, we have that \Cref{eq:def_matrix_M}
is equivalent to
\begin{equation}
\label{eq:transformed}
  \rk
  \begin{pmatrix}
    A' & \\
    B' & M \cdot \nabla f (P) \\
    C' & \\
  \end{pmatrix}
  = 1 \,.
\end{equation}
Now we consider the polynomial $M \cdot f$ and we observe that the chain rule gives
\begin{gather*}
  \partial_x (M\cdot f) = \partial_x \bigl( f(M^{-1} \ \prescript{t} {} (x \ y \ z)) \bigr) = \prescript{t} {}(M^{-1})^{(1)}(\nabla f) \bigl( M^{-1}\ \prescript{t} {} (x \ y \ z) \bigr) \,, \\
  \partial_y (M\cdot f) = \partial_y \bigl( f(M^{-1} \ \prescript{t} {} (x \ y \ z)) \bigr) = \prescript{t} {}(M^{-1})^{(2)}(\nabla f) \bigl( M^{-1}\ \prescript{t} {} (x \ y \ z) \bigr) \,, \\
  \partial_z (M\cdot f) = \partial_z \bigl( f(M^{-1} \ \prescript{t} {} (x \ y \ z)) \bigr) = \prescript{t} {}(M^{-1})^{(3)}(\nabla f) \bigl( M^{-1}\ \prescript{t} {} (x \ y \ z) \bigr) \,,
\end{gather*}
where $(M^{-1})^{(j)}$ denotes the $j$-th column of the matrix $M^{-1}$. Thus
\[
  \nabla (M \cdot f) = \prescript{t} {} M^{-1} \cdot (\nabla f) \bigl(M^{-1}\ \prescript{t} {} (x \ y \ z)\bigr) \,,
\]
so we have
\[
  \nabla (M \cdot f)(Q)=\nabla (M \cdot f)(M \cdot P)=
  \prescript{t} {} M^{-1} \cdot (\nabla f) (M^{-1}\ M \cdot P)=\prescript{t} {} M^{-1} \cdot (\nabla f)(P) \,.
\]
Finally, if we choose $M \in \SO_3(\C)$, we have
$\prescript{t} {} M^{-1}=M$. We deduce that
\Cref{eq:transformed} holds if and only if $Q \in \Eig{M\cdot f}$, so the statement is proved.
\end{proof}

\section{Conditions imposed by aligned eigenpoints}
\label{conditions}

Having a given point as eigenpoint determines linear conditions on the space of ternary cubic forms,
which can hence be encoded into a matrix, called the \emph{matrix of conditions} (\Cref{definition:matrix_conditions}).
We analyze the possible ranks of matrices of conditions relative to three to five points, with one or two aligned triples.

\subsection{The matrix of conditions}

We begin to explore the condition that $P\in \p^2$ is an eigenpoint of a ternary cubic; such linear conditions are encoded in
a $3 \times 10$ matrix, which we now introduce.

\begin{definition}
\label{definition:matrix_conditions}
Consider $\p^9 = \p(\C[x,y,z]_3)$, the space of all ternary cubics.
Throughout this paper, we consider the standard monomial basis for $\C[x,y,z]_3$ and we set~$\mathcal{B}$ to be the following vector
\begin{equation}
\label{eq:vector_basis}
  \mathcal{B} := (x^3, x^2 y, x y^2, y^3, x^2 z, x y z, y^2 z, x z^2, y z^2, z^3)
  \in \bigl( \C[x,y,z]_3 \bigr)^{\oplus 10} \,.
\end{equation}
For $f \in \C[x,y,z]_3$, denote by~$[f]$ the corresponding point in~$\p^9$; we denote by~$w_f$ the (column) vector of coordinates of~$f$, and we will use the same notation for the projective coordinates of~$[f]$.
For a point $P \in \p^2$ with coordinates $(A: B: C)$, the condition that $P$ is an eigenpoint,
i.e., that $g_1(P) = g_2(P) = g_3(P) = 0$ where the $g_i$ are as in \Cref{eq:def_minors},
can be expressed as three linear conditions in the coordinates of~$\p^9$,
hence in the form
\[
  \Phi(P) \cdot w_f = 0 \,,
\]
where $\Phi(P)$ is a $3 \times 10$ matrix with entries depending on $A, B, C$.
The matrix $\Phi(P)$ is called the \emph{matrix of conditions} imposed by~$P$.
We denote by~$\phi_1(P)$, $\phi_2(P)$, and~$\phi_3(P)$ the rows of~$\Phi(P)$.
Written as vectors, they are
\begin{equation}
\label{eq:matrix_conditions_rows}
  \begin{aligned}
    \phi_1(P) &=
    \scalemath{0.9}{
    (-3A^2B, A(A^2 - 2B^2), B(2A^2 - B^2), 3AB^2,
     -2ABC, C(A^2 - B^2), 2 ABC,
     -B C^2, A C^2, 0)} \,, \\
    \phi_2(P) &=
    \scalemath{0.9}{
    (-3A^2 C,
     -2ABC,
     -CB^2,
     0,
     A(A^2-2C^2),
     B(A^2 - C^2),
     AB^2,
     C(2A^2-C^2),
     2ABC,
     3AC^2)} \,,\\
    \phi_3(P) &=
    \scalemath{0.9}{
    (0,
     -A^2C,
     -2ABC,
     -3CB^2,
     A^2 B,
     A(B^2 - C^2),
     B(B^2-2C^2),
     2ABC,
     C(2B^2-C^2),
     3BC^2)} \,.
  \end{aligned}
\end{equation}
More generally, if $P_1, \dotsc, P_n$ are points in the plane, we denote by~$\Phi(P_1, \dotsc, P_n)$ the matrix whose rows are
\[
  \phi_1(P_1), \phi_2(P_1), \phi_3(P_1),
  \dotsc,
  \phi_1(P_n), \phi_2(P_n), \phi_3(P_n),
\]
and we call it the
\emph{matrix of conditions} imposed by $P_1, \dotsc, P_n$.
\end{definition}

\begin{definition}
Given a matrix $M$ of type $m \times 10$, we denote by $\Lambda(M) \subset \p^9$ the projective linear system of cubics $[a_0 x^3 + \dotsb + a_9 z^3]$ such that $M \, \prescript{t}{}{\left( a_0 \,  \cdots \,  a_9 \right)} = 0$.
\end{definition}

The description of~$\Phi(P)$ can be shortened as follows: consider the vectors $\de_x \mathcal{B}$,
$\de_y \mathcal{B}$, and~$\de_z \mathcal{B}$;
if $P=(A: B: C)$, we get:
\begin{equation}
\label{eq:vector_conditions}
  \begin{aligned}
    \phi_1(P) &= A\cdot \de_y \mathcal{B}(P) - B\cdot \de_x \mathcal{B}(P) \,, \\
    \phi_2(P) &= A\cdot \de_z \mathcal{B}(P) - C\cdot \de_x \mathcal{B}(P) \,, \\
    \phi_3(P) &= B\cdot \de_z \mathcal{B}(P) - C\cdot \de_y \mathcal{B}(P) \,.
  \end{aligned}
\end{equation}
\begin{rmk}
\label{remark:rank_2}
By analyzing the entries of \Cref{eq:matrix_conditions_rows}, it is not difficult to check that the rank of~$\Phi(P)$ is never $\leq 1$. \Cref{eq:vector_conditions} gives the relation
\begin{equation}
\label{eq:syzygy}
  C \, \phi_1(P) - B \, \phi_2(P) + A \, \phi_3(P) = 0 \,.
\end{equation}
Therefore, the vectors~$\phi_1(P)$, $\phi_2(P)$, and~$\phi_3(P)$ are linearly dependent and the matrix~$\Phi(P)$ has rank~$2$.
As a consequence, if $P_1, \dots, P_n$ are $n$ points of the plane, we have:
\[
  \rk \,\Phi(P_1, \dots, P_n) \leq \min \left\{2n, 10 \right\} \,.
\]
\end{rmk}

\subsection{Possible ranks of the matrix of conditions}

In what follows, we want to study the possible values of the rank of the matrix
$\Phi(P_1, \dots, P_n)$ for several configurations of points $P_1, \dots, P_n$
(and several values of~$n$).
In particular, we study the ideal~$J_k$ of order $k$ minors of the
involved matrices of conditions and we deduce some bounds about the rank from the primary
decomposition of~$J_k$.
Most of these computations are done with the aid of a computer algebra system.
Nevertheless, in many cases, the result cannot be reached just by brute force,
but it is necessary to preprocess the ideal~$J_k$.
In particular, it is often convenient to first saturate the ideal~$J_k$ with respect to
the condition that some points are distinct or that three of them are not aligned
(when this is the case).
Another important simplification that we adopt sometimes, makes use
of the action of~$\SO_3(\C)$: thanks to it, we can assume that one of
the point is either $(1: 0: 0)$ or $(1: \iii: 0)$; see \Cref{proposition:two_orbits}.

We start with the following lemma, which is extremely useful
to speed up the computations.

\begin{lemma}
\label{lemma:minors}
Let $l_1 < \cdots <l_n$ be $n$ indices (where $3 \leq n \leq 10$) and let $P = (A: B: C)$ be a point of the plane.
Construct three $1 \times n$ matrices $w_1$, $w_2$, $w_3$ by extracting the entries of position $l_1, \dotsc, l_n$ from~$\phi_1(P)$, $\phi_2(P)$, and~$\phi_3(P)$, respectively. If $L$ is a $(n-2) \times n$ matrix, set:
\[
  L_1 := \left( \begin{array}{c} w_1 \\ w_2 \\ L \end{array} \right), \quad
  L_2 := \left( \begin{array}{c} w_1 \\ w_3 \\ L \end{array} \right), \quad
  L_3 := \left( \begin{array}{c} w_2 \\ w_3 \\ L \end{array} \right)
\]
Then
\[
  B \det(L_1) = A \det(L_2), \quad
  C \det(L_1) = A \det(L_3), \quad
  C \det(L_2) = B \det(L_3),
\]
therefore $(A: B: C) = \bigl( \det(L_1): \det(L_2): \det(L_3) \bigr)$.
\end{lemma}
\begin{proof}
The statement easily follows from the equality $C w_1 - B w_2 + A w_3 = 0$, which is a direct consequence of \Cref{eq:syzygy}.
\end{proof}

Next, we point out a property of the lines that are tangent to the isotropic conic~$\iso$.
First of all, we fix some notation that is used throughout the paper.

\begin{definition}
\label{definition:sigma}
For $P_1 = (A_1: B_1: C_1)$ and $P_2 = (A_2: B_2: C_2)$, we set
\begin{equation}
\label{formula:sigma}
\begin{aligned}
  \sigma(P_1, P_2) &= \scl{P_1}{P_1} \scl{P_2}{P_2} - \scl{P_1}{P_2}^2 \\
   &= s_{11}s_{22}-s_{12}^2
\end{aligned}
\end{equation}
This is a bihomogeneous polynomial of bidegree~$(2,2)$ on $\p^2 \times \p^2$.
\end{definition}

\begin{rmk}
\label{remark:sigma_discr}
The form~$\sigma$ is the discriminant of the intersection between the line~$P_1 \vee P_2$ and~$\iso$.
Indeed, if $u_1 P_1 + u_2P_2$ is a generic point on the line $P_1 \vee P_2$, the relation
\[
  \scl{u_1 P_1 + u_2 P_2}{u_1 P_1 + u_2 P_2} =0.
\]
gives the intersection $(P_1 \vee P_2) \cap \iso$.
The discriminant of the latter (as a polynomial in $u_1$ and $u_2$) is $4 \bigl( \scl{P_1}{P_2}^2 - \scl{P_1}{P_1} \scl{P_2}{P_2} \bigr)$, and this proves the claim.
\end{rmk}

\begin{prop}
\label{proposition:sigma_tangency}
Let $P_1$, $P_2$ be two distinct points in the plane and let $r=P_1 \vee P_2$.
Then the following are equivalent:
\begin{enumerate}
  \item $\sigma(P_1, P_2) = 0$;
  \item $\sigma(Q_1, Q_2) = 0$ for all pairs of distinct points $Q_1, Q_2 \in r$;
  \item the line~$r$ is tangent to~$\iso$ at some point;
  \item there exists a point $T \in r$ such that $T \in \iso$ and $\scl{T}{Q} = 0$ for all $Q \in r$, $Q \neq T$.
\end{enumerate}
\end{prop}
\begin{proof}
By \Cref{remark:sigma_discr}, we have that $\sigma(P_1, P_2) = 0$ if and only if $r$ is tangent to~$\iso$ in a point~$T$; this shows that the first three items are equivalent.

Now, if $\sigma(P_1, P_2) = 0$, then $r$ is tangent to~$\iso$ at a point~$T$, hence $\scl{T}{T} = 0$.
Moreover, by the second item $\sigma(T, Q) = 0$ for all $Q \in r$ with $Q \neq T$, and by the definition of $\sigma$ the claim follows.
The converse is immediate.
\end{proof}

\begin{prop}
\label{proposition:three_distinct_ranks}
Let $P_1, P_2, P_4$ be three distinct points of the plane. Then:
\begin{itemize}
  \item $5 \leq \rk \,\Phi(P_1, P_2, P_4) \leq 6$;
  \item
  $\rk \, \Phi(P_1, P_2, P_4) = 5$ if and only if $P_1, P_2, P_4$
  are aligned and the line joining them is tangent to~$\iso$
  in one of the three points.
\end{itemize}
\end{prop}
\begin{proof}
The rank of $\Phi(P_1, P_2, P_4)$ cannot be larger than~$6$ by \Cref{remark:rank_2}.
\nb{03}{F1} gives that $\Phi(P_1, P_2, P_4)$ cannot have rank $4$ (since its minors of order~$5$ never vanish all at the same time) and it also proves the second item.
\end{proof}

\begin{prop}
\label{proposition:three_aligned_plus_one}
Let $P_1, P_2, P_3, P_4$ be four distinct points of the plane such that
$P_1, P_2, P_3$ are aligned, let $r = P_1 \vee P_2 \vee P_3$ and assume $P_4 \not \in r$.

If $\rk \,\Phi(P_1, P_2, P_3, P_4) \leq 7$ then $r$ is tangent to~$\iso$ in one of the three points~$P_1$, $P_2$, and~$P_3$.
\end{prop}
\begin{proof}
\nb{03}{F2} shows that the condition imposed by the vanishing of all order $8$ minors of the matrix~$\Phi(P_1, P_2, P_3, P_4)$ is equivalent to the statement.
\end{proof}

\subsection{The conditions \texorpdfstring{$\delta_1$}{delta1}, \texorpdfstring{$\bar{\delta}_1$}{deltabar1}, and \texorpdfstring{$\delta_2$}{delta2}}

We now define three expressions depending on the homogenous coordinates of a triple or on a $5$-tuple of points in the plane.
Such expressions will be crucial in describing what happens when we have aligned eigenpoints.

\begin{definition}
\label{definition:delta1}
We define a multihomogeneous polynomial of multidegree~$(2,1,1)$ on $\p^2 \times \p^2 \times \p^2$:
if $P_i = (A_i: B_i: C_i)$ for $i \in \{1, 2, 4\}$, we set
\begin{align*}
  \delta_1(P_1, P_2, P_4) &:=
  \scl{P_1}{P_1} \scl{P_2}{P_4} - \scl{P_1}{P_2}\scl{P_1}{P_4} =
  \scl{P_1\times P_2}{P_1 \times P_4} \\
  &\phantom{:}= s_{11} s_{24}-s_{12}s_{14} \,,
\end{align*}
where $\times$ denotes the cross product, see \Cref{eq:cross_product}.
\end{definition}

\begin{rmk}
\label{remark:delta1_meaning}
Geometrically, the condition $\delta_1(P_1, P_2, P_4) = 0$ corresponds to the orthogonality of the lines~$P_1 \vee P_2$ and~$P_1 \vee P_4$, if the three points are not collinear, while
$\delta_1(P_1, P_2, P_4) = 0$ implies that $\sigma (P_1,P_2)=0$ if they are collinear.
\end{rmk}

\begin{definition}
\label{definition:delta1b}
Let $P_1$, $P_2$ and~$P_3$ be distinct aligned points in the plane.
We define the polynomial
\begin{align*}
  \overline{\delta}_1(P_1, P_2, P_3) &:=
  \scl{P_1}{P_1} \scl{P_2}{P_3} + \scl{P_1}{P_2}\scl{P_1}{P_3} \\
  &\phantom{:}= s_{11} s_{24}-s_{12}s_{14}\,.
\end{align*}
\end{definition}

\begin{definition}
\label{definition:Vconf}
Let $P_1, P_2, P_3, P_4, P_5$ be five distinct points of the plane
such that $P_1, P_2, P_3$ and $P_1, P_4, P_5$ are aligned, and
$P_1,P_2,P_4$ not aligned.
We call such a configuration a \emph{$V$- configuration}.
\end{definition}

\begin{definition}
Let $P_1, \dots, P_5$ be a $V$- configuration.
We define the polynomial
\begin{align*}
  \delta_2(P_1, P_2, P_3, P_4, P_5) &:=
  \scl{P_1}{P_2} \scl{P_1}{P_3} \scl{P_4}{P_5} -
  \scl{P_1}{P_4} \scl{P_1}{P_5} \scl{P_2}{P_3} \\
  &\phantom{:}= s_{12}s_{13}s_{45}-s_{14}s_{15} s_{23} \,.
\end{align*}
\end{definition}

The next results give a geometric description of the zero loci of the expressions just defined.
\begin{lemma}
\label{lemma:characteristics_d1_d2}
It holds
\begin{align}
\label{lemma_delta_case1}
  \delta_1(P_1, P_2, P_4) = 0 \mbox{ iff } &\scl{P_4}{s_{11}P_2-s_{12}P_1} = 0\\
  \mbox{iff } &\scl{P_2}{s_{11}P_4-s_{14}P_1} = 0 \,. \nonumber
\end{align}
\begin{align}
\label{lemma_delta_case2}
  \overline{\delta}_1(P_1, P_2, P_3) = 0 \mbox{ iff } &
  P_1 \mbox{ is on~$\iso$ and } P_1 \vee P_2 \vee P_3 \mbox{ is tangent to~$\iso$ in $P_1$, or} \\
  & P_3 = (s_{12}^2+s_{11}s_{22}) \, P_1 - 2s_{11}s_{12} \, P_2 \mbox{ or} \nonumber \\
  & P_2 = (s_{13}^2+s_{11}s_{33}) \, P_1 - 2s_{11}s_{13} \, P_3 \nonumber
\end{align}
\end{lemma}
\begin{proof}
 The proof is a symbolic verification; one can find it in \nb{03}{F3}.
\end{proof}

In general, if we fix $P_1,P_2,P_4,P_5$, the condition
$\delta_2=0$ uniquely determines $P_3$. However, there are some exceptions, which are listed in the following proposition.

\begin{prop}
\label{proposition:definitionP3}
For five points $P_1, \dots, P_5$ in a $V$- configuration, it holds that
$\delta_2(P_1, \dotsc, P_5) = 0$ if and only if (up to a permutation of $P_2, \dots, P_5$) at least one of the following conditions
is satisfied:
\begin{enumerate}
  \item $s_{12} = 0$ and $s_{14} = 0$;
  \label{defP3_1}
  \item $s_{12} = 0$ and $s_{22} = 0$;
  \label{defP3_2}
  \item $\sigma(P_1, P_2) = 0$ and $\sigma(P_1, P_4) = 0$;
  \label{defP3_3}
  \item $P_3 = (s_{14}s_{15}s_{22}-s_{12}^2s_{45})P_1  +s_{12}(s_{11}s_{45}-s_{14}s_{15})P_2$.
  \label{defP3_4}
\end{enumerate}
\end{prop}

\begin{proof}
Condition~(\ref{defP3_4}), when defined (i.e., when the coefficients of~$P_1$ and~$P_2$ are not zero), easily comes from the definition of~$\delta_2$.
The point~$P_3$ is not defined by that formula if and only if
$s_{14}s_{15}s_{22}-s_{12}^2s_{45}=0$ and $s_{11}s_{12}s_{45}-s_{12}s_{14}s_{15}=0$.
Hence, we study the ideal generated by these two polynomials, together with the
polynomial $s_{12}s_{13}s_{45}-s_{14}s_{15} s_{23}$ which defines~$\delta_2$ in terms of the $s_{ij}$. The corresponding
ideal decomposes into several ideals which are, up to a permutation of
the indices: $(s_{12}, s_{14})$, $(s_{12}, s_{22})$ and
\[
  J = (s_{13}s_{22} - s_{12}s_{23}, s_{14}s_{15} - s_{11}s_{45}, s_{12}s_{13} -
  s_{11}s_{23}, s_{12}^2 - s_{11}s_{22}) \,.
\]
Then we give generic coordinates to the five
points and substitute them into the ideal $J$. It is possible to see that
$J$ and the ideal $\bigl(\sigma(P_1, P_2), \sigma(P_1, P_4)\bigr)$ are equal (up to
saturations w.r.t. the condition that all points are distinct and that $P_1, P_2, P_4$ are not aligned).
These computations are carried in \nb{03}{F4}.
\end{proof}
\begin{rmk}
In the first case of \Cref{proposition:definitionP3}, the corresponding cubic
is studied in \Cref{further_alignments}, configuration $(C_5)$.
In the second case, the conditions $s_{12}=0$ and $s_{22}=0$ imply $\sigma(P_1, P_2) = 0$ and $P_2\in \iso$, so the line~$P_1 \vee P_2 \vee P_3$ is tangent to~$\iso$ in~$P_2$
and this case was considered in~\Cref{proposition:three_distinct_ranks}; in particular,
the matrix $\Phi(P_1, P_2, P_3)$ has rank~$5$, so
$\Phi(P_1, \dots, P_5)$ has rank $\le 9$, regardless of
$P_4$ and~$P_5$.
The third case gives either that all the
points of the two lines of the $V$- configuration are eigenpoints (in case
$P_2 \not\in \iso$ or $P_4 \not\in \iso$) or the matrix $\Phi(P_1, \dots, P_5)$
has rank~$8$ (if $P_2, P_4 \in \iso$); see \Cref{rank_8} and \Cref{positive_dim}, respectively.
The last case is the generic one and expresses $P_3$ in terms of the remaining four points.
\end{rmk}
\begin{prop}
\label{proposition:d1d2}
Let $P_1, \dots, P_5$ be a $V$- configuration. Then
\[
  \rk \,\Phi(P_1, \dots, P_5) \leq 9
  \quad \mbox{if and only if} \quad
  \delta_1(P_1, P_2, P_4) \cdot \delta_2(P_1, \dots, P_5) = 0 \,.
\]
\end{prop}
\begin{proof}
Let
\[
  P_1 = (A_1: B_1: C_1) \,, \quad
  P_2 = (A_2: B_2: C_2) \,, \quad
  P_4 = (A_4: B_4: C_4) \,,
\]
and $P_3 = u_1 \, P_1 + u_2 \, P_2$ and $P_5 = v_1 \, P_1 + v_2 \, P_4$ for some $u_1, u_2, v_1, v_2$.
We find, via symbolic computation, that the determinant of
the submatrix of the matrix of conditions of $P_1, \dotsc, P_5$ obtained by selecting the first two rows from each $\Phi(P_i)$ for $i \in \{1, \dotsc, 5\}$ is
\begin{gather*}
  A_1A_2A_4(u_1A_1+u_2A_2)(v_1A_1+v_2A_4) \cdot u_1^2u_2^2v_1^2v_2^2 \cdot D^5 \cdot
  \delta_1(P_1,P_2,P_4) \cdot \delta_2(P_1,\dots,P_5) \,,
\end{gather*}
where $D$ is the determinant of the matrix whose rows are $P_1, P_2, P_4$.
As a consequence of \Cref{lemma:minors}, the order~$10$ minors of the matrix of conditions are polynomials of the form
\begin{gather*}
  X_1X_2X_4 X_3 X_5 \cdot u_1^2u_2^2v_1^2v_2^2 \cdot D^5 \cdot
  \delta_1(P_1,P_2,P_4) \cdot \delta_2(P_1,\dots,P_5) \,,
\end{gather*}
where $X_i$ varies among all coordinates of~$P_i$ for $i \in \{1, \dotsc, 5\}$.
Since each of $P_1, \dotsc, P_5$ has at least one non-zero coordinate,
and $D$ is non-zero, as well as are non-zero $u_1, u_2, v_1, v_2$ (since we assume that $P_1, P_2, P_4$ are not aligned and the points are distinct), we have that
$\rk \, \Phi(P_1, \dots, P_5) \leq 9$ if and only if
\[
  \delta_1(P_1, P_2, P_4) \cdot \delta_2(P_1, \dots, P_5) = 0 \,. \qedhere
\]
\end{proof}

\begin{lemma}
\label{lemma:special_case_rank_8}
Let $P_1, \dots, P_5$ be a $V$- configuration and assume that
$
  s_{12} =
  s_{22} =
  s_{14} =
  s_{44} = 0 \,.
$
Then the matrix $\Phi(P_1, \dots, P_5)$ has rank~$8$.
\end{lemma}
\begin{proof}
By \Cref{proposition:sigma_tangency},
the lines~$P_1 \vee P_2$ and~$P_1 \vee P_4$ are tangent to~$\iso$ in~$P_2$ and~$P_4$, respectively.
The point~$P_1$ cannot be on~$\iso$, hence, using the
action of~$\SO_3(\C)$, we can assume $P_1 = (1: 0: 0)$.
Since every element of~$\SO_3(\C)$ leaves $\iso$ invariant,
when we transform the point~$P_1$ into $(1: 0: 0)$,
we transform the points~$P_2$ and~$P_4$ into, respectively,
the points $(0: \iii: 1)$ and $(0: -\iii: 1)$ (which are the common points to~$\iso$ and the tangent lines through~$P_1$).
Therefore, it is enough to study the
specific configuration of the points:
\begin{gather*}
  P_1 = (1: 0: 0) \,, \quad P_2=(0: \iii: 1) \,, \quad P_3=(u_1, \iii u_2, u_2) \,, \\
  P_4 = (0: -\iii: 1) \,, \quad P_5 = (v_1, -\iii v_2, v_2) \,,
\end{gather*}
where $(u_1: u_2), (v_1: v_2) \in \p^1$.
\nb{03}{F5} proves that the matrix of conditions of this configuration has rank~$8$.
\end{proof}

The following result gives a complete description of the possible ranks of matrix of conditions for $V$- configurations.

\begin{theorem}
\label{theorem:rank_V}
Let $P_1, \dots, P_5$ be a $V$- configuration. Then we have:
\begin{enumerate}
  \item $8 \leq \rk \,\Phi(P_1, \dots, P_5) \leq 10$\,;
  \item $\rk \,\Phi(P_1, \dots, P_5) \leq 9$ if and only if
  $\delta_1(P_1, P_2, P_4) \cdot \delta_2(P_1, \dots, P_5) =0$\,;
  \item $\rk \,\Phi(P_1, \dots, P_5) = 8$ if and only if
  \begin{itemize}
    \item $\delta_1(P_1, P_2, P_4) = 0$, \
    $\overline{\delta}_1(P_1, P_2, P_3) = 0$, \
    $\overline{\delta}_1(P_1, P_4, P_5) = 0$\,; or
    \item the line~$P_1 \vee P_2$ is tangent to~$\iso$ in~$P_2$ or~$P_3$
    and the line~$P_1 \vee P_4$ is tangent to~$\iso$ in~$P_4$ or~$P_5$;
    moreover, in this case we have $\delta_1(P_1, P_2, P_4) \neq 0$.
  \end{itemize}
  In particular, in both cases $\delta_2(P_1, P_2, P_3, P_4, P_5) = 0$ holds.
\end{enumerate}
\end{theorem}
\begin{proof}
If the rank is $\leq 7$, from
\Cref{proposition:three_aligned_plus_one} applied to $P_1, P_2, P_3, P_4$ and $P_1, P_4, P_5, P_2$,
the lines~$P_1 \vee P_2$ and~$P_1 \vee P_4$ are tangent to~$\iso$ (the first in $P_2$ or $P_3$ and the second in $P_4$ or $P_5$).
Then, by \Cref{proposition:sigma_tangency} and \Cref{lemma:special_case_rank_8} we get a contradiction.
This shows the first item.

The second item is \Cref{proposition:d1d2}.

We are left to proving the third item. By \Cref{lemma:special_case_rank_8} and \Cref{proposition:d1d2}, each of the two conditions implies that the rank is~$8$.
\nb{03}{F6} proves the converse; in particular, it shows that if $P_1 \in \iso$, then $\rk \,\Phi(P_1, \dots, P_5) \geq 9$. In the second case, it is necessarily $\delta_1(P_1,P_2,P_4)
\neq 0$; indeed, in the case $P_2, P_4 \in \iso$, then
\[
  \scl{P_2 +P_4}{P_2+P_4} =
  \scl{P_2}{P_2} + \scl{P_4}{P_4}+2\scl{P_2}{P_4} =
  2\scl{P_2}{P_4} \neq 0 \,,
\]
as the point~$P_2 + P_4$ is different from both~$P_2$ and~$P_4$, and the line~$P_2 \vee P_4$ has no other intersection points with~$\iso$. It follows that
$\delta_1 (P_1,P_2,P_4)=s_{11}s_{24} \neq 0$, as $P_1 \not\in \iso$. The cases $P_3 \in \iso$, respectively $P_5\in \iso$, are similar.
Therefore, if $\rk \, \Phi(P_1, \dots, P_5) = 8$, then one of the two conditions above must hold.
\end{proof}

\section{\texorpdfstring{$V$}{V}- configurations}
\label{V-configurations}

In order to get a $V$- configuration, there are two possible constructions:
\begin{itemize}
    \item[(a)] Recalling \Cref{lemma:characteristics_d1_d2}, it is quite easy to construct five points $P_1, \dots, P_5$ that are in a $V$- configuration
and such that $\delta_1(P_1, P_2, P_4)= 0$: the points~$P_1$
and~$P_2$ can be taken in an arbitrary way, $P_4$ has to be chosen in such
a way that it satisfies \Cref{lemma_delta_case1}
and~$P_3$ and~$P_5$ have to be chosen on the lines~$P_1 \vee P_2$ and~$P_1 \vee P_4$,
respectively. In particular, the corresponding locus of cubic curves
has dimension~$7$.
The construction of a random cubic
of five points as above, gives a smooth cubic curve whose $7$ eigenpoints
do not have other collinearities (in addition to those of a
$V$- configuration).
    \item[(b)] If we want a $V$- configuration that satisfies the condition
$\delta_2(P_1, \dots, P_5) = 0$, we choose $P_1$, $P_2$, $P_4$ arbitrarily;
we choose $P_3$ on the line~$P_1 \vee P_2$ and~$P_5$ on the line~$P_1 \vee P_4$ so that $P_3$ satisfies \Cref{proposition:definitionP3}, \Cref{defP3_4}.
Also in this case, the locus of cubics is of dimension~$7$.
We will show that in this situation there are further alignments.
\end{itemize}

\subsection{\texorpdfstring{$V$}{V}- configurations of rank~\texorpdfstring{$9$}{9}}
\label{rank_9}

In this section, we consider five points $P_1, \dots, P_5$ in a $V$- configuration and we assume that $\rk \, \Phi(P_1, \dots, P_5) = 9$.
We show that if $\delta_2(P_1, \dotsc, P_5) = 0$ and the unique corresponding cubic has a regular eigenscheme, then the two other eigenpoints $P_6$ and $P_7$ are aligned with $P_1$.

The following two results are instrumental to the symbolic computations in the proof of the main result.

\begin{lemma}
\label{lemma:construct_cubic}
Let $\mathcal{H}$ be a $9 \times 10$-submatrix of rank~$9$ of $\Phi(P_1, \dots, P_5)$.
Let $\mathcal{H}_i$ be the minor of~$\mathcal{H}$ given by deleting the $i$-th column ($i=1, \dots, 10$).
Then a polynomial defining the unique cubic~$C=V(f)$ with $P_1, \dotsc, P_5 \in \Eig{f}$ is
\[
  f(X) = \sum_{i=1}^{10}(-1)^i\det(\mathcal{H}_i)\cdot \mathcal{B}_i
  = \det \left(
  \begin{array}{c} \mathcal{H} \\ \mathcal{B} \end{array}
  \right) \,.
\]
where $\mathcal{B}$ is the monomial vector as in \Cref{eq:vector_basis}.
\end{lemma}
\begin{proof}
 The statement follows from Cramer's rule for the resolution of linear systems of maximal rank.
\end{proof}

The matrix~$\mathcal{H}$ allows one also to compute the generators of the ideal of the eigenscheme.

\begin{prop}
\label{proposition:geiser1}
The three minors given in \Cref{eq:def_minors} are proportional to
\[
  \det \left(
  \begin{array}{c} \mathcal{H} \\ \phi_1(X) \end{array}
  \right),\quad
  \det \left(
  \begin{array}{c} \mathcal{H} \\ \phi_2(X) \end{array}
  \right), \quad
  \det \left(
  \begin{array}{c} \mathcal{H} \\ \phi_3(X) \end{array}
  \right)
\]
\end{prop}
\begin{proof}
We have
\begin{align*}
  g_1 & = x \cdot \de_y f(X)- y \cdot \de_x f(X)  =
  x \cdot \de_y \det \left(
  \begin{array}{c} \mathcal{H} \\ \mathcal{B} \end{array}
  \right) - y \cdot
  \de_x \det \left(
  \begin{array}{c} \mathcal{H} \\ \mathcal{B} \end{array}
  \right) \\
  & = \det \left(
  \begin{array}{c} \mathcal{H} \\ x \cdot \de_y \mathcal{B} - y \cdot \de_x \mathcal{B} \end{array}
  \right)  = \det \left(
  \begin{array}{c} \mathcal{H} \\ \phi_1(X) \end{array}
  \right)
\end{align*}
and similarly for $g_2 = x \cdot \de_z f(X)- z \cdot \de_x f(X)$
and $g_3 = y \cdot \de_z f(X)- z \cdot \de_y f(X)$.
\end{proof}

\begin{prop}
\label{proposition:G_split}
Assume that $P_1, \dots, P_5 $ satisfy $\delta_2(P_1, \dotsc, P_5) = 0$ and $\rk \, \Phi(P_1, \dotsc, P_5) = 9$. Then the cubic curve
of equation $A_1 g_3 - B_1 g_2 + C_1 g_1=0$ contains the lines~$P_1 \vee P_2$ and~$P_1 \vee P_4$, where $g_1, g_2, g_3$ are given in \Cref{eq:def_minors}.
\end{prop}
\begin{proof}
The equation of the line~$P_1 \vee P_2$ can be expressed in the following way
\begin{equation}
\label{eq:lineP1P2}
  \left\langle P_1 \times P_2, (x,y,z) \right\rangle = 0 \,;
\end{equation}
since $P_1 \vee P_2$ is contracted by the Geiser map~$\gamma_{\Eig{f}}$,
and by the description of the fibers given in \Cref{eq:fibers}, such a line is contained in the conic of equation
\[
  \left\langle P_1 \times P_2, (\de_x f, \de_y f, \de_z f) \right\rangle = 0 \,.
\]
As a consequence, for any point $\overline{P} = (\bar x: \bar y: \bar z)$ of the line from \Cref{eq:lineP1P2}, the following equations are satisfied:
\[
  \left\{
  \begin{array}{l}
    \left\langle P_1 \times P_2, P_1 \right\rangle = 0 \,,\\[2pt]
    \bigl\langle P_1 \times P_2, \overline{P} \bigr\rangle = 0 \,,\\[2pt]
    \bigl\langle P_1 \times P_2, \nabla f (\overline{P}) \bigr\rangle = 0 \,.
  \end{array}
  \right.
\]
It follows that the linear system in the variables $X,Y,Z$
\[
  \left\{
  \begin{array}{l}
    \bigl\langle P_1, (X,Y,Z) \bigr\rangle = 0 \,,\\[2pt]
    \bigl\langle \overline{P}, (X,Y,Z) \bigr\rangle = 0 \,,\\[2pt]
    \bigl\langle \nabla f (\overline{P}),
    (X,Y,Z) \bigr\rangle = 0 \,.
  \end{array}
  \right.
\]
admits $P_1 \times P_2$ as a non-zero solution,
thus the determinant of the $3 \times 3$ coefficient matrix with rows
 $P_1$, $\overline{P}$, $\nabla f (\overline{P})$ is zero. This is equivalent to saying that $A_1 g_3 - B_1 g_2 + C_1 g_1$ vanishes at all the points~$\overline{P}$ of the line~$P_1 \vee P_2$.

We argue similarly for line~$P_1 \vee P_4$.
\end{proof}
\begin{prop}
\label{proposition:third_alignment}
Let $P_1, \dots, P_5$ be a $V$- configuration of five points such that
$\rk\, \Phi(P_1, \dots, P_5) = 9$ and
$\delta_2(P_1, \dots, P_5) = 0$. Let $V(f)$ be
the cubic with $P_1, \dots, P_5$
eigenpoints and suppose $\Eig{f}$ is
regular. Then the two other eigenpoints are aligned with~$P_1$.
\end{prop}
\begin{proof}
According to \Cref{proposition:G_split}, the polynomial
$C_1g_1-B_1g_2+A_1g_3$ splits into three linear factors $r_1$, $r_2$, and~$r_3$, which
correspond to the line~$P_1 \vee P_2$, the line~$P_1 \vee P_4$, and the line~$P_6 \vee P_7$. Using \Cref{proposition:geiser1}, Notebook \nb{04}{F1} shows that~$P_1\in P_6 \vee P_7$.
\end{proof}

\subsection{\texorpdfstring{$V$}{V}-configurations of rank~\texorpdfstring{$8$}{8}}
\label{rank_8}
In this section, we study the possible configurations of
eigenpoints for cubics with a $V$- configuration
and $\rk \, \Phi(P_1, \dots, P_5) = 8$. According
to~\Cref{theorem:rank_V}, we have to distinguish two cases
\begin{gather}
  \delta_1(P_1, P_2, P_4)=\overline{\delta}_1(P_1, P_2, P_3) =
  \overline{\delta}_1(P_1, P_4, P_5) = 0 \,, \text{or}
  \label{rk8_1} \\
  \sigma(P_1, P_2) = \sigma(P_1, P_4) = 0 \ \ \mbox{and} \ \ s_{22} = s_{44} = 0 \,;
  \label{rk8_2}
\end{gather}
the latter condition means that \emph{the two lines of
the $V$- configuration are tangent to $\iso$ in $P_2$ and~$P_4$}.
To start, we characterize when, in a $V$- configuration
$P_1, \dotsc, P_5$, the point~$P_1$ is singular.

\begin{prop}
\label{proposition:P1_sing}
Let $P_1, \dots, P_5$ be a $V$- configuration and let
$C = V(f)$ be a cubic with
$P_1, \dots, P_5\in \Eig{f}$. Then $P_1$ is
singular for~$C$ if
and only if $P_1$ is not on $\iso$ and one of the following conditions
is satisfied:
\begin{enumerate}
  \item $\delta_1(P_1, P_2, P_4) = 0$ and $\overline{\delta}_1(P_1, P_2, P_3) = 0$;
  \item $\delta_1(P_1, P_2, P_4) = 0$ and $\overline{\delta}_1(P_1, P_4, P_5) = 0$;
  \item $\overline{\delta}_1(P_1, P_2, P_3) = 0$ and
  $\overline{\delta}_1(P_1, P_4, P_5) = 0$.
\end{enumerate}
\end{prop}
\begin{proof}
See Notebook \nb{04}{F2}.
\end{proof}

\begin{prop}
\label{proposition:char_rank_8}
  Let $P_1, \dotsc, P_5$ be a $V$- configuration such that \Cref{rk8_1} holds.
  Then $P_4$ is orthogonal to $s_{11} \, P_2 - s_{12} \, P_1$ and, up to swapping $2 \leftrightarrow 3$ and $4 \leftrightarrow 5$, it holds
  \[
   P_3 = (s_{12}^2+s_{11}s_{22}) \, P_1 - 2s_{11}s_{12} \, P_2 \,, \quad
   P_5 = (s_{14}^2+s_{11}s_{44}) \, P_1 - 2s_{11}s_{14} \, P_4 \,.
  \]
\end{prop}
\begin{proof}
  The statement about~$P_4$ follows from \Cref{lemma_delta_case1}.
  From \Cref{lemma_delta_case2}, we obtain
  \[
  \left\{ \!\!
  \scalemath{0.9}{
  \begin{array}{ll}
  (L_1) & \left\langle P_1, P_1 \right\rangle = \sigma(P_1, P_2) = 0 \,, \text{ or} \\
  (L_2) & P_3 = (s_{12}^2+s_{11}s_{22}) \, P_1 - 2s_{11}s_{12} \, P_2 \,, \text{ or} \\
  (L_3) & P_2 = (s_{13}^2+s_{11}s_{33}) \, P_1 - 2s_{11}s_{13} \, P_3
  \end{array}
  }
  \right.
  \;\; \text{and} \;\;
  \left\{ \!\!
  \scalemath{0.9}{
  \begin{array}{ll}
  (R_1) & \left\langle P_1, P_1 \right\rangle = \sigma(P_1, P_4) = 0 \,, \text{ or} \\
  (R_2) & P_5 = (s_{14}^2+s_{11}s_{44}) \, P_1 - 2s_{11}s_{14} \, P_4 \,, \text{ or} \\
  (R_3) & P_4 = (s_{15}^2+s_{11}s_{55}) \, P_1 - 2s_{11}s_{15} \, P_5
  \end{array}
  }
  \right.
  \]
  Conditions $(L_1)$ and $(R_1)$ are not compatible: indeed, together they imply $P_1 \vee P_2 = P_1 \vee P_4$.
  Notebook \nb{04}{F3} shows that $(L_1)$ is not compatible with $(R_2)$ or $(R_3)$ and $(R_1)$ is not compatible with $(L_2)$ or $(L_3)$.
  The possibilities that are left show the statement.
\end{proof}

\begin{rmk}
From \Cref{proposition:char_rank_8}, it follows that if we choose
\begin{itemize}
  \item $P_1$ and $P_2$ in an arbitrary way;
  \item $P_4$ in the $\p^1$
  space of points orthogonal to $s_{11} \, P_2 - s_{12} \, P_1$;
  \item $P_3 = (s_{12}^2+s_{11}s_{22}) \, P_1 - 2s_{11}s_{12} \, P_2$;
  \item $P_5 = (s_{14}^2+s_{11}s_{44}) \, P_1 - 2s_{11}s_{14} \, P_4$;
\end{itemize}
the matrix $M = \Phi(P_1, \dots, P_5)$ has rank $8$
and therefore $\dim \Lambda(M) = 1$ and the dimension of the variety
of the corresponding cubics is $6$.
Moreover, from~\Cref{proposition:P1_sing} all these cubics
are singular in $P_1$.
If we take a random point of this variety, the corresponding
cubic has $7$ distinct eigenpoints (and is as expected, singular in $P_1$) and the
$7$ points do not satisfy other collinearities in addition to those of the
$V$- configuration.
In particular, there is no alignment $(P_1, P_6, P_7)$.
Since \Cref{rk8_1} implies $\delta_2(P_1, \dots, P_5) = 0$, the hypothesis ``rank~$9$'' in~\Cref{proposition:third_alignment} is necessary.
An example can be found in \nb{04}{F4}.
\end{rmk}

\begin{rmk}
\label{remark:particular_cases}
In general, a cubic having eigenpoints $P_1, \dots, P_5$ in a $V$- configuration satisfying \Cref{rk8_1} has only the two alignments $(P_1, P_2, P_3)$ and $(P_1, P_4, P_5)$ among its eigenpoints;
in particular, it does not have the collinearity $(P_1, P_6, P_7)$.
Since \Cref{rk8_1} implies $\delta_2(P_1, \dots, P_5) = 0$, the hypothesis ``rank $9$'' in~\Cref{proposition:third_alignment} is necessary.
\end{rmk}

Concerning possible sub-cases with further collinearities of the points,
the following results hold:

\begin{prop}
\label{proposition:three_d_three_alignments}
If $P_1, \dots, P_5$ satisfy \Cref{rk8_1},
then, in the pencil $\Lambda \bigl( \Phi(P_1, \dotsc, P_5)\bigr)$ there is
a cubic curve with $7$ eigenpoints with the following three alignments:
\[
 (P_1, P_2, P_3) \,, \quad (P_1, P_4, P_5) \,, \quad \text{and} \quad (P_1, P_6, P_7) \,.
\]
No choices of $P_1, \dots, P_5$ allow one to obtain further alignments of the $7$ eigenpoints.
\end{prop}
\begin{proof}
As observed in the proof of~\Cref{theorem:rank_V}, if $P_1 = (1: \iii: 0)$,
the matrix
$\Phi(P_1, \dots, P_5)$ cannot have rank smaller than $9$, so the only
case to consider is $P_1 = (1: 0: 0)$.
Notebook \nb{04}{F5} provides the proof of the statement.
\end{proof}
Moreover, we have:
\begin{prop}
\label{proposition:d2_6align}
If the five points $P_1, \dots, P_5$ satisfy \Cref{rk8_1}
and if we impose the condition that there is an eigenpoint, say $P_6$, aligned with $P_2$ and~$P_4$, then the eigenpoints satisfy all these
alignments:
\[
  (P_1, P_2, P_3), (P_1, P_4, P_5),
  (P_2, P_4, P_6), (P_2, P_5, P_7),
  (P_3, P_4, P_7), (P_3, P_5, P_6) \,.
\]
Hence the points~$P_6$ and~$P_7$ are determined by $P_1, \dots, P_5$
since
$P_6 = (P_2 \vee P_4) \cap (P_3 \vee P_5)$
and $P_7 = (P_3 \vee P_4) \cap (P_2 \vee P_5)$.
A similar result holds if we take $P_3$ in place of~$P_2$ or~$P_5$
in place of~$P_4$.
\end{prop}
\begin{proof}
We define the points $P_1, \dots, P_5$ as in the previous proposition, then
we consider a point of the line $P_2 \vee P_4$ and we impose that it is
an eigenpoint, i.e., we impose that $\rk \, \Phi(P_1, \dotsc, P_6) < 10$.
This request defines a point~$P_6$ and it turns out that $P_6$ is also aligned
with~$P_3$ and~$P_5$. We then define $P_7$ as the point
$(P_2 \vee P_5) \cap (P_3 \vee P_4)$. It turns out that
$\rk \, \Phi(P_1, \dots, P_7) < 10$, so $P_7$ is an eigenpoint. The seven
points $P_1, \dotsc, P_7$ satisfy the above collinearities.
This is shown in Notebook \nb{04}{F6}.
\end{proof}

\Cref{proposition:three_d_three_alignments} and \Cref{proposition:d2_6align} exhaust all the possible configurations
of collinearities in case of \Cref{rk8_1}.
Now we consider the \Cref{rk8_2}.
A generic cubic which satisfies \Cref{rk8_2} is
a cubic of the one-dimensional linear system
$\Lambda\bigl(\Phi(P_1, \dotsc, P_5)\bigr)$, where:
\begin{itemize}
  \item $P_1$ is any point of the plane (not on $\iso$);
  \item $P_2$ and $P_4$ are the two tangency points to $\iso$ given by the tangents from~$P_1$;
  \item $P_3$ is any point on the line~$P_1 \vee P_2$ (different from~$P_1$ and~$P_2$)
  and $P_5$ is any point on the line~$P_1 \vee P_4$ (different from~$P_1$ and~$P_4$).
\end{itemize}
The variety of all the cubics with a $V$- configuration
that satisfies condition~(\ref{rk8_2}) is therefore five-dimensional.
The five points satisfy the condition $\delta_2(P_1, P_2, P_3, P_4, P_5) = 0$.

The reciprocal position of the eigenpoints of the cubics of this family
is described by the following:
\begin{prop}
\label{proposition:rk8_2B}
The generic cubic of the family of cubics satisfying \Cref{rk8_2}
has seven eigenpoints with the alignments:
\[
  (P_1, P_2, P_3), \ (P_1, P_4, P_5), \ (P_1, P_6, P_7)
\]
Among these points we have the relation
$\scl{P_1 \times P_6}{P_3\times P_5}=0$
(i.e., the lines~$P_1 \vee P_6$ and~$P_3 \vee P_5$ are orthogonal).
In the family there is a sub-family of cubics whose eigenpoints have the following alignments:
\[
  (P_1, P_2, P_3),\ (P_1, P_4, P_5),\ (P_1, P_6, P_7),\ (P_2, P_4, P_6).
\]
In this case the points~$P_6$ and~$P_7$ are given by the formulas:
\begin{equation}
\label{eq:formulaeP6_P7}
P_6 = s_{15}s_{34}\, P_2 + s_{13}s_{25}\, P_4, \quad
P_7 = s_{15}(s_{26}s_{46}+s_{24}s_{66})\, P_1+ s_{11}s_{24}s_{56}\, P_6
\end{equation}
and a sub-family whose eigenpoints have the following
alignments:
\[
  (P_1, P_2, P_3),\ (P_1, P_4, P_5), \
  (P_1, P_6, P_7),\ (P_2, P_5, P_6), \
  (P_3, P_4, P_6),\ (P_3, P_5, P_7)
\]
In this case, the point~$P_6$ (given, for instance, by the formula
$P_6 = s_{15} \, P_3 + s_{13} \, P_5$) is obtained as
the intersection~$(P_2 \vee P_5) \cap (P_3 \vee P_4)$ and consequently
$P_7 = (P_1 \vee P_6) \cap (P_3 \vee P_5)$.\\
No other collinearities among the eigenpoints are possible.
\end{prop}
\begin{proof}
By considering again the action of~$\SO_3(\C)$, we can assume:
\[
  P_1 = (1: 0: 0), \quad
  P_2 = (0: \iii: 1), \quad
  P_4 = (0: -\iii: 1)
\]
and $P_3 = u_1 \, P_1 + u_2 \, P_2$, $P_5 = v_1 \, P_1 + v_2 \, P_4$.
In this situation, it
is easy to see from \Cref{lemma:construct_cubic} that $\Lambda\bigl(\Phi(P_1, \dots, P_5)\bigr)$
is the following pencil of cubic forms:
\[
  f(l_1, l_2) = l_1 \, f_1 + l_2 \, f_2
\]
where $l_1, l_2$ are parameters and
\begin{align*}
 f_1 & = x \cdot \left(2x^{2} + 3 y^{2} + 3 z^{2}\right)\\
  f_2 & = (y + \iii z) \cdot (y - \iii z)
  \cdot \bigl(2 \iii x u_{2} v_{2} + y (u_{2} v_{1}- u_{1} v_{2})
  - \iii z (u_{2} v_{1} + u_{1} v_{2})\bigr)
\end{align*}
The explicit expression $f(l_1, l_2)$ allows us to
construct the ideal of the seven eigenpoints and (after
saturations w.r.t.\ the condition that the five points $P_1, \dotsc, P_5$ are distinct),
we get an ideal generated by a line~$r$ and a conic~$\Gamma$, whose zeros are
the points~$P_6$ and~$P_7$. Since the line~$r$ contains the point~$P_1$,
the points $P_1, P_6, P_7$ are collinear and $r$ results orthogonal to~$P_3 \vee P_5$.
In order to see if there are further collinearities among the
eigenpoints, we have to distinguish three cases (up to permutation
of the indices): $P_2, P_4, P_6$ are collinear or $P_2, P_5, P_6$ are
collinear, or $P_3, P_5, P_6$ are collinear. In the first case, $P_6$
is the point $r \cap (P_2\vee P_4)$. If it is an eigenpoint, it must be
on $\Gamma$, hence $l_1$ and~$l_2$ have a specific value which gives a sub-family
of~$f(l_1, l_2)$ and we can compute the explicit coordinates of
all the seven eigenpoints of the cubics of this family. The other two
cases can be studied in a similar way. See \nb{04}{F7}.
\end{proof}

\begin{rmk}
\label{remark:three_orthog} For further references, it is perhaps worth noting that the explicit
construction of the
points in~\Cref{proposition:d2_6align} and~\Cref{proposition:rk8_2B} allows one to verify
several relations among them. In particular, the seven eigenpoints
of~\Cref{proposition:d2_6align} are such that
\[
\scl{P_1\times P_2}{P_1 \times P_4}=0, \quad
\scl{P_2\times P_4}{P_3 \times P_5}=0, \quad
\scl{P_3\times P_4}{P_2 \times P_5}=0;
\]
and it holds:
\[
P_4 = (P_2\times P_3)s_{25}s_{35}-s_{23}(P_2\times P_5)s_{35}+ s_{23}s_{25}(P_3\times P_5).
\]
In the first sub-family described in~\Cref{proposition:rk8_2B} we have
$P_1 = P_2 \times P_4$, so
\[
\scl{P_1\times P_2}{P_2 \times P_4}=0, \quad
\scl{P_1\times P_4}{P_2 \times P_4}=0, \quad
\scl{P_1\times P_6}{P_2 \times P_4}=0.
\]
In the second family:
\[
\scl{P_1\times P_4}{P_3 \times P_4}=0, \quad
\scl{P_1\times P_2}{P_2 \times P_5}=0, \quad
\scl{P_3\times P_5}{P_1 \times P_6}=0,
\]
hence it holds:
\[
P_3 = (P_1 \times P_5)s_{16}s_{56}-s_{15}(P_1\times P_6)s_{56}+s_{15}s_{16}(P_5 \times P_6) \,.
\]
\end{rmk}

\section{The locus of cubic ternary forms with an aligned triple of eigenpoints}
\label{locus_one_alignment}

The results of the previous sections allow one to determine the dimension and the degree of the locus of cubics having
at least one aligned triple of eigenpoints.

\begin{definition}
\label{definition:locus_L}
Let $\sU \subset \p^9$ be the following locus:
\[
  \sU:= \{[f]\in \p^9 \setminus \Delta_{3,3} \ | \ \Eig{f} \ \textrm{contains \ an \ aligned \ triple}\}, \,
\]
where $\Delta_{3,3}$ denotes the eigendiscriminant of \Cref{definition:eigendiscriminant},
and define $\sL \subseteq \p^9$ as the closure of~$\sU$:
\[
  \sL := \overline \sU \subset \p^9 \,.
\]
\end{definition}

\begin{theorem}
\label{theorem:irreducible}
The variety~$\sL$ is an irreducible hypersurface.
\end{theorem}

\begin{proof}
We observe that, as there exist
polynomials $f$ such that $\Eig{f}$ has no aligned triple, we have $\dim \sL \le 8$.

Consider the symmetric product $(\p^2)^{(3)}$ and set $\mathcal{AL} \subset (\p^2)^{(3)}$ to be the locus of unordered triples of distinct aligned points. Observe that $\mathcal{AL}$ is irreducible of dimension $5$ as its closure is the hypersurface given by the vanishing of the determinant of the order~$3$ matrix of the coordinates of three general points. Alternatively, the closure of the locus $\mathcal {AL}$ can be seen as the symmetric quotient of the projective line bundle $\sF \subset \p^2 \times \p^2 \times \p^2$ parametrized by the triples $(P_1, P_2, u_1 P_1 +u_2P_2)$, where $(u_1:u_2) \in \p^1$.

If we set
\[
  \sU':= \{[f]\in \p^9 \setminus \Delta_{3,3} \ | \ \Eig{f} \ \text{contains exactly $1$ aligned triple}\} \,,
\]
we have a surjective morphism
\[
  \alpha \colon \mathcal{U}' \to \mathcal{AL} \,,
\]
assigning to each $[f] \in \mathcal{U}'$ the unique triple of eigenpoints.
By \Cref{proposition:three_distinct_ranks}, the fibers of~$\alpha$ are projective linear systems of dimension~$3$ over the irreducible open subset
\[
  \mathcal{W} := \mathcal{AL}
  \setminus \{(P_1,P_2,P_3) \in \mathcal{AL}
  \ | \ \sigma(P_1,P_2)=0, s_{11} s_{22} s_{33}=0\}
\]
of aligned triples not lying
on a tangent line to the isotropic conic~$\iso$ and with tangency point one of the~$P_i$'s.
By the Fiber Dimension Theorem, this implies that the open subset $\alpha^{-1} (\mathcal{W}) \subset \sU$ is irreducible of dimension~$8$, so $\dim \sL =8$.

To prove the irreducibility, we observe that since the eigenpoint condition corresponds to two linear conditions on the coefficients of cubic polynomials, by linear algebra methods we can determine a local parametrization of~$\sL$.
Specifically, given a general aligned triple $(P_1, P_2, u_1 P_1 +u_2P_2)$, the matrix of conditions $M=\Phi(P_1, P_2, u_1 P_1 +u_2P_2)$ has rank $6$, so the associated linear system
$\Lambda \bigl( \Phi(P_1, P_2, u_1 \, P_1 + u_2 \, P_2) \bigr)$ has dimension~$3$. We can then express $6$ coefficients of the general polynomial
$[f]\in \Lambda \bigl( \Phi(P_1, P_2, u_1 \, P_1 + u_2 \, P_2) \bigr)$ as rational functions of $P_1,P_2,u_1,u_2$ and of the remaining $4$ coefficients. The free coefficients depend on the position of a non zero minor of order six.
We claim that the last $3$ columns of the matrix
$M$ are always in the linear span of the first $7$ columns.

To prove the claim, set $c_0, \dots, c_9$ to be the $10$ columns of~$M$.
Moreover,
if the three components of~$P_1 \times P_2$ are called $\alpha, \beta, \gamma$, we set:
\[
  N_1 = \left(
  \begin{array}{ccc}
    \alpha & 0 & 0 \\
    0 & \beta & 0\\
    0 & 0 & \gamma
  \end{array}
  \right), \quad
  N_2 = \left(
  \begin{array}{ccc}
    0 & \alpha & 0 \\
    \gamma & 0 & 0\\
    0 & 0 & \beta
  \end{array}
  \right).
\]
The six
columns $c_0, c_1, c_2, c_4, c_5, c_7$ of~$M$ are linearly dependent. Indeed, if $L_1$ is the $9\times 3$ matrix whose columns are
$c_0, c_2, c_7$ and $L_2$ is given by the columns $c_1, c_4, c_5$,
then we have:
\begin{equation}
  (L_1 N_1 + 2 L_2N_2) (P_1 \times P_2) = 0 \,;
  \label{eq:lin_comb}
\end{equation}
the linear combination of $c_0, c_1, c_2, c_4, c_5, c_7$ which is zero
is obtained by expanding this expression. This computation is contained in \nb{05}{F1}.

Similarly, the columns $c_1, c_2, c_3, c_5, c_6, c_8$ are linearly dependent
and~\eqref{eq:lin_comb} holds if $L_1$ in this case is $[c_1, c_3, c_8]$ and
$L_2$ is $[c_2, c_5, c_6]$. Finally, the columns
$c_4, c_5, c_6, c_7, c_8, c_9$ are linearly dependent and~\eqref{eq:lin_comb}
holds if we take $L_1 = [c_4, c_6, c_9]$ and $L_2 = [c_5, c_7, c_8]$.

As a consequence, local parametrizations of~$\sL$ can be given by considering the following open subsets:
for any multiindex
\[
  I = \{i_1, \dots, i_6\} \subset \{0, 1, \dots, 6\} \,,
\]
we set
\begin{multline*}
  \sV_I :=
  \bigl\{
    [(P_1, P_2, u_1 P_1 +u_2P_2)] \in \mathcal{AL} \ | \ \text{the\ columns\ of\ } M \
    \text{relative \ to} \ I\ \text{are\ independent}
  \bigr\} \,.
\end{multline*}
Then for any $[(P_1, P_2, u_1 \, P_1 + u_2 \, P_2)] \in \sV_I$, by expressing any element
$[f] =[\mathcal{B} \cdot w_f]\in \Lambda \bigl( M \bigr)$ with
$ w_f = (b_0,
  b_1,
  b_2,
  b_3,
  b_4,
  b_5, b_6,b_7,b_8,b_9)$,
if $i\in I$ the coefficients $b_i$ depend on the free parameters $b_i=b_i(P_1,P_2,u_1,u_2,b_j,b_7,b_8,b_9)$, where $j \not \in I$, $0\le j \le 6$. Therefore, we can interpret $w_f=w_f(P_1,P_2,u_1,u_2,b_j,b_7,b_8,b_9)$ as a suitable rational function.
This leads to the definition of the rational map
\[
  \alpha_I \colon \sV_I \times \p^3 \to \sL, \quad
  \alpha_I (P_1,P_2,u_1,u_2,b_j,b_7,b_8,b_9)=
  [\mathcal{B} \cdot w_f(P_1,P_2,u_1,u_2,b_j,b_7,b_8,b_9)] \,.
\]
Next we observe that for any $I \subset \{0,\dots, 6\}$, the image
\[
  W_I := \alpha_I (\sV_I \times \p^3)
\]
is irreducible, being image of an irreducible variety, and it
contains an open subset. Indeed, assume for simplicity (the other cases are similar) that
$I=\{0,1,2,3,4,5\}$; by suitably adapting the argument of \Cref{lemma:construct_cubic}, we see that an element $[f]\in W_I$ is represented by the determinant of the
matrix given by the rows
$0,1,3,4,6,7$ of $M$ and the additional $4$ rows
\[
  \left(
  \begin{array}{cccccc}
    0 & \cdots & 1&0&0&-b_6 \\
    0 & \cdots & 0&1&0&-b_7 \\
    0 & \cdots & 0&0&1&-b_8 \\
    & & & \mathcal{B} & & \\
  \end{array}
  \right) \,.
\]
In particular, the coefficient $b_9$ of~$z^3$ is equal to the determinant of the order~$6$ minor of $M$ relative to the rows $0,1,3,4,6,7$ and the first $6$ columns. It follows that $W_I \supset \sU \cap \{ a_9 \neq 0\}$.

Finally, we claim that the irreducible subvarieties $W_I$ for $I\subset \{0,\dots, 6\}$ have a common non-empty intersection
including internal points. Indeed, as one can verify with an explicit symbolic computation, an example of a polynomial
class~$[g]$ with an aligned triple of eigenpoints and satisfying
\[
  [g] \in \bigcap_I W_I \setminus \Delta_{3,3}
\]
is given by a random choice of~$P_1$, $P_2$ and $(u_1:u_2)$ and $b_7,b_8,b_9$.

Hence $\sU$ is covered by irreducible open subsets, each intersecting every other, so it is irreducible, and the same holds for its closure $\sL$.
\end{proof}

Next we want to determine the degree of the hypersurface $\sL$. We shall need the following.

\begin{prop}
The locus $\sV \subset \sL \subset \p^9$ of cubic forms with two or more aligned triples of eigenpoints has dimension
$\dim \sV = 7$.
\end{prop}
\begin{proof}
In $(\p^2)^5$, we consider the locus $\sW$ of $V$- configurations $(P_1,P_2,P_3,P_4,P_5)$. By a dimension count it is $\dim \sW=8$.
By \Cref{theorem:rank_V}, the five points are eigenpoints if and only if $\delta_1 (P_1,P_2,P_4) \cdot
\delta_2 (P_1,P_2,P_3,P_4,P_5)=0$. The two equations define two divisors in $\sW$, so the locus $\sE\sW$ of $V$- configurations constituted by eigenpoints has dimension $7$.

Finally, there is an open subset of $\sE\sW$ where the rank of the condition matrix is~$9$, hence $\sE\sW$ is birational to~$\sV$.
\end{proof}

As a consequence, we have the following result.

\begin{corollary}
If $[f],[g] \in \p^9 \setminus \Delta_{3,3}$ are general cubics,
then the general cubic in the pencil
$\lambda f + \mu g$ for $(\lambda: \mu) \in \p^1$ has no aligned triples of eigenpoints, and there is a finite number of cubics having exactly one aligned triple.
\end{corollary}

If $g_1, g_2, g_3$ are the three minors of \Cref{eq:def_matrix} relative
to a cubic form~$f$, we denote by
\[
  \Sigma_f := \p \bigl( \left\langle g_1, g_2, g_3 \right\rangle \bigr),
\]
the net of cubics, whose base locus is the eigenscheme~$\Eig{f}$.

\begin{lemma}
\label{lemma:scroll}
If $f$ and $g$ are general cubics, then
\[
  \mathcal{N} := \bigcup_{(\lambda : \mu) \in \p^1} \Sigma_{\lambda f + \mu g} \subset \p^9
\]
is an embedding of a rational projective bundle and has degree~$3$.
\end{lemma}
\begin{proof}
Consider the projective bundle given by the family of planes
\[
  \mathcal{P} := \{ \Sigma_{\lambda f + \mu g} \, : \, (\lambda: \mu) \in \p^1 \} \subset \p^1 \times \p^9
\]
Observe that we can assume that $\Sigma_{\lambda f + \mu g} \cong \p^2$
for every $(\lambda:\mu) \in \p^1$. Indeed, it fails to be a net if and only if $g_1,g_2,g_3$ are linearly dependent, and this happens if and only if the three partials are linearly dependent. The latter condition is satisfied if and only if $V(\lambda f + \mu g)$ is a set of concurrent lines. In this case $[\lambda f + \mu g]$ belongs to a $5$ - dimensional locus.

Then $\mathcal{N}$ is the projection of~$\mathcal{P}$ on the second factor.
However, the map $\mathcal{P} \to \mathcal{N}$ contracts no subvariety of any plane of~$\mathcal{P}$, so either it is an embedding or it contracts some horizontal curve. In the latter case, all the planes of the family should intersect in at least one point. In particular, the two nets $\Sigma_f$ and~$\Sigma_g$ should have non-empty intersection.
If we denote by~$g_1$, $g_2$ and~$g_3$ the $2 \times 2$ minors relative to~$f$, and by~$h_1$, $h_2$ and~$h_3$ the ones relative to~$g$, the vectorial dimension of the linear span $\left\langle g_1, g_2, g_3, h_1, h_2, h_3 \right\rangle$ should be strictly less than $6$. This can be avoided, since such a condition corresponds to a proper closed subscheme of~$\p^9 \times \p^9$.

It follows that if $f$ and~$g$ are general enough, then $\mathcal{N}$ is a $3$-dimensional rational normal scroll in
\[
  \mathcal{N} \subset \p(\left\langle g_1, g_2, g_3, h_1, h_2, h_3 \right\rangle) \cong \p^5.
\]
Being a variety of minimal degree, its degree is $5+1-3 = 3$ by the classical result of \cite{EH}.
\end{proof}

\begin{theorem}
The degree of~$\sL$ is equal to
\[
  \deg \ \sL = 15 \,.
\]
\end{theorem}

\begin{proof}
We start by observing that a reduced $0$-dimensional eigenscheme contains an aligned triple if and only if the net of cubics
$\Sigma_f$ contains a cubic which splits in three lines, a so called \emph{triangle}. Moreover, if $f$ is general enough, we have exactly one aligned triple and the other $4$ points are in general position; in this case, the net~$\Sigma_f$ contains exactly three triangles, namely the unions of the line passing through the aligned triple and the reducible conics through the $4$ points in general position.

To determine the degree of~$\sL$, we consider a general pencil of cubic forms $\lambda f + \mu g$, and we compute the number of elements with associated net $\Sigma_{\lambda f + \mu g}$ containing a triangle.

To this aim, denote by $\mathcal{T} \subset \p^9$ the variety of triangles; it is a classical result that its dimension is~$6$ and its degree is~$15$,
see for instance \cite[Section~2.2.2]{3264}. We now consider the variety~$\mathcal{N}$ from \Cref{lemma:scroll}.
Note that, since each net containing a triangle, actually contains exactly $3$ of them, the number of nets of~$\mathcal{N}$ containing some triangle is given by
\[
  \frac{\mathcal{T} \cdot \mathcal{N}}{3} = \frac{{15} \cdot {3}}{3} = 15 \,.
\]
This implies that $\deg \sL = 15$.
\end{proof}

\section{Eigenschemes of positive dimension}
\label{positive_dim}

In this section, we consider positive dimensional eigenschemes. By \cite{BGV}, an eigenscheme cannot be of pure dimension~$1$, and the possible $1$-dimensional components are of degree~$1$ or~$2$.
In what follows, we shall determine the degree of the
$0$-dimensional residual scheme in both cases.

\begin{prop}
\label{proposition:positive_dimension}
Let $C = V(f) \subset \p^2$ be a cubic curve.
Assume that $\dim \Eig{f} = 1$.
\begin{enumerate}
  \item If the $1$-dimensional component of~$\Eig{f}$ is a line~$\ell$,
  then the residual subscheme $Z := \mathrm{Res}_{\ell} \bigl( \Eig{f} \bigr)$ in~$\Eig{f}$ with respect to~$\ell$ has degree~$3$. Moreover, the scheme $Z$ is not contained in a line.
  \item If the $1$-dimensional component of~$\Eig{f}$ is a conic~$\Gamma$,
  then the residual subscheme $Z := \mathrm{Res}_{\Gamma} \bigl( \Eig{f} \bigr)$ in~$\Eig{f}$ with respect to~$\Gamma$ has degree~$1$.
\end{enumerate}
\end{prop}

The proof relies on the exactness of the Koszul complex associated with the regular sequence~$x$, $y$, and~$z$, specifically we use the following result (see \cite[Theorem~7.3.13]{Dolgachev}).

\begin{lemma}
\label{lemma:Koszul}
Let $h_1,h_2,h_3\in\C[x,y,z]_d$ with $d \ge 1$. Then
\begin{equation}
\label{eq:linear_relation}
  zh_1-yh_2+xh_3 = 0
\end{equation}
if and only if there exist $m_1,m_2,m_3\in\C[x,y,z]_{d-1}$ such that
\begin{equation}
\label{eq:minors_lemma}
  h_1 = xm_2-ym_3 \,, \qquad
  h_2 = zm_3-xm_1 \,, \qquad
  h_3 = ym_1-zm_2 \,.
\end{equation}
\end{lemma}

\begin{proof}
If $h_1,h_2,h_3$ satisfy \Cref{eq:minors_lemma}, then it is immediate to check that they satisfy \Cref{eq:linear_relation} as well.
Conversely, assume that \Cref{eq:linear_relation} holds and let $R = \C[x,y,z]$.
The Koszul complex in the ring~$R$ is an exact sequence of $R$-modules
\[
  0 \to R\xrightarrow{\alpha} R^{\oplus 3} \xrightarrow{\beta} R^{\oplus 3} \xrightarrow{\gamma} R \to R/(x,y,z) \to 0 \,,
\]
where the maps are $\alpha(p) = (p x, p y, p z)$, $\gamma (w_1,w_2,w_3) = w_1 x + w_2 y + w_3 z$ and $\beta$ is defined by the matrix
\[
  \left(
  \begin{array}{ccc}
    0 & -z & y\\
    z & 0 & -x\\
    -y & x & 0 \\
  \end{array}
  \right) \,.
\]
The syzygy $zh_1-yh_2+xh_3=0$ implies that $(h_3, -h_2, h_1)$ is in the kernel of~$\gamma$,
and since the Koszul complex is exact, the triple $(h_3,-h_2, h_1)$ lies in the image of~$\beta$.
It follows that there exist $m_1,m_2,m_3 \in \C[x,y,z]_{d-1}$ such that $\beta (m_3,m_2,m_1)=(h_3,-h_2, h_1)$,
so \Cref{eq:minors_lemma} holds.
\end{proof}

Now we are in the position to prove \Cref{proposition:positive_dimension}.

\begin{proof}[Proof of \Cref{proposition:positive_dimension}]
Let $g_1$, $g_2$ and~$g_3$ be the order~$2$ minors determining the eigenscheme of~$f$, and let $g$ be the greatest common factor.
By writing
\[
  g_i = g \, h_i, \quad i=1,2,3
\]
we have that the residual scheme is defined by the ideal
$(h_1,h_2,h_3)$. Moreover, the linear
syzygy between the generators~$g_i$ gives rise to the syzygy:
\[
  z\, h_1 - y\, h_2 + x\, h_3 = 0 \,.
\]
By \Cref{lemma:Koszul}, the triple $(h_1,h_2,h_3)$ is the triple of order two minors of a matrix
\[
  \begin{pmatrix}
    x & y & z \\
    m_3 & m_2 & m_1
  \end{pmatrix} \,.
\]
for suitable forms $m_i \in \C[x,y,z]_r$, where $r =2 - \deg g \ge 0$.
By the assumption that $g$ is the greatest common factor of the three minors $g_1$, $g_2$ and~$g_3$, the zero scheme $Z$ of $(h_1,h_2,h_3)$ is $0$-dimensional, and
by \Cref{theorem:nonempty} its degree is~$3$ if $r=1$. Observe that $Z$, being intersection of three conics with no common component, is not contained in a line.

Finally, if $r=0$, the triple $(h_1,h_2,h_3)$ corresponds to three linear forms belonging to a pencil, thus the zero locus is a point.
\end{proof}

\begin{es}
Consider the form
\[
  f(x, y, z) = x^2 (y - z) \,.
\]
In the language of \Cref{proposition:positive_dimension} and its proof, we have $\ell=x$,
\[
 h_1 = x^2-2y^2+2yz \,, \quad
 h_2 = -x^2-2yz+2z^2\,, \quad
 h_3 =-x(y+z) \,.
\]
The two syzygies in degree~$3$ are:
\[
  z \, h_1 - y \, h_2 + x \, h_3 = 0, \quad
  xh_1 \, +x\,h_2 - 2(y+z) \, h_3 = 0.
\]
Finally, $Z = \{ (0:1:1),(2:1:-1),(-2:1:-1) \}$.
Observe that one point is on the singular line $x=0$.
\end{es}

\subsection{Eigenschemes containing a line}
Here we study the cases in which $\Eig{f}$ contains a line. We shall see that
this condition is equivalent to the condition of having four collinear
eigenpoints; this will allow us to charaterize the cubics $C = V(f)$ which have a line
in $\Eig{f}$ and, finally, to prove that the polynomials $f$
belong all to the locus $\sL \subset \p^9$ (see \Cref{definition:locus_L}).
We first need the following result, which will be generalized in Section 6.2.

\begin{lemma}
\label{lemma:twoTangentsCiso} Let $r = ax+by+cz$ be a line of the plane
and suppose it intersects $\iso$ in two
distinct points~$P_1$ and~$P_2$. Consider the cubic with eq:
\begin{equation}
\label{eq:2_lines_of_eigenpoints}
  f(r) = \left( r^2-3\left(a^2+b^2+c^2\right)\iso \right) \, r \,;
\end{equation}
then, the two tangent lines to $\iso$ in~$P_1$ and~$P_2$
are contained in $\Eig {f(r)}$.
\end{lemma}
\begin{proof}
A generic point on $\iso$ is of the form $(\lambda^2 + \mu^2,
\iii\lambda^2 -\iii\mu^2, 2\iii\lambda \mu)$ for $(\lambda: \mu) \in \p^1$.
We can therefore easily obtain two points~$P_1$ and~$P_2$ on~$\iso$,
determine the line $r = P_1 \vee P_2$ and define~$f(r)$ as above.
The two tangents to~$\iso$ in~$P_1$ and~$P_2$ turn out to be lines of eigenpoints
for~$f(r)$. For details, consult \nb{06}{F1}.
\end{proof}

\begin{lemma}
\label{lemma:four_points_on_line}
Suppose that $P_1, P_2, P_3, P_4$ are four distinct points belonging to a line~$t$.
Given a cubic $C=V(f)$, then $P_1, \dotsc, P_4 \in \Eig{f}$ if and only if
$t \subseteq \Eig{f}$.
Moreover,
\begin{equation*}
  6 \leq \rk \,\Phi(P_1, P_2, P_3, P_4) \leq 7
\end{equation*}
and the rank is~$6$ if and only if $\sigma(P_1, P_2) = 0$, i.e.\ if
and only if $t$ is tangent to the isotropic conic.
\end{lemma}
\begin{proof}
The computations are in \nb{06}{F2}.
\end{proof}

We are now in the position to give a characterization of cubics with an eigenscheme containing a line.
\begin{prop}
\label{proposition:eigenline_tangent}
Let $t$ be a line of $\mathbb{P}^2$.
\begin{itemize}
   \item
   If $t$ is not tangent to the isotropic conic, then $t \subseteq \Eig{f}$ for a cubic $C=V(f)$ if and
only $f = t^2\ell$, where $\ell$ is any line of the plane.
  \item
  If $t$ is tangent to the isotropic conic~$\iso$ (in a point~$P$),
then $t \subseteq \Eig{f}$ for a cubic $C=V(f)$ if and only if
  \begin{equation}
  \label{eq:cubics_with_tangent_eigenline}
    f = t^2 \ell+\lambda f(r_0),
  \end{equation}
  where $\ell$ is any line of the plane, $\lambda \in \C$,
  $r_0\neq t$ is any fixed line passing through $P$
  and~$f(r_0)$ is defined by \Cref{eq:2_lines_of_eigenpoints}.
  \end{itemize}
  Moreover, if $t$ is tangent to $\iso$, any cubic of
  \Cref{eq:cubics_with_tangent_eigenline} is singular in $P$ and, if $V(f)$ is irreducible, then $V(f)$ is nodal and the line $t$ belongs to the tangent cone of $V(f)$ in $P$.
\end{prop}
\begin{proof}
Suppose that $t$ is not tangent to $\iso$ and let $P_1, \dots, P_4$ be four
distinct points on it. By \Cref{lemma:four_points_on_line}, the projective linear system
$\Lambda(\Phi(P_1, \dotsc, P_4))$ is two-dimensional, but also all the cubics
defined by $t^2\ell$ (for $\ell$ any line of $\mathbb{P}^2$) have $t$ among
their eigenpoints and form a
two-dimensional linear system of cubics, so the two linear systems
coincide.

If $t$ is tangent to $\iso$ in a point $P$, w.l.o.g.\ we can assume that
$P$ is $(1: \iii: 0)$ and $t$ is $x+\iii y$. If we impose to the generic
cubic of $\mathbb{P}^2$ to contain $t$ in the eigenscheme, we get a linear
system of cubics of dimension $3$. It can be shown that a basis of such a linear system is given by $H_1 = t^2x$, $H_2 = t^2y$, $H_3 = t^2z$ and
$H_4 = f(r_0)$, where $r_0$ is any line passing through $P$ and different from $t$. This proves the
second claim. It is immediate to see that $P$ is singular for the cubics of the linear system. \nb{06}{F3} contains the details of the computations which allow to conclude.
\end{proof}
\begin{prop}
\label{proposition:limitCubics}
Any cubic that has a line~$t$ in the eigenscheme is the limit of a family of cubics whose general member has a $0$-dimensional eigenscheme with an aligned triple.
\end{prop}
\begin{proof}
If the line~$t$ is not tangent to $\iso$, we can assume it is the line $z=0$. We fix on it the three points
\[
  P_1= (1: 0: 0), \ P_2 = (0: 1: 0), \ P_3 = (1: 1: 0)
\]
and we consider the three dimensional linear system of cubic forms $\Lambda \bigl( \Phi(P_1, P_2, P_3) \bigr)$, whose elements $f$ are of the form
\[
  f = (ax + by + cz)z^2 + d(x^3+y^3), \quad a, b, c, d \in \C \,.
\]
If $d=0$, we have $f=t^2 \ell$, hence, by \Cref{proposition:eigenline_tangent}, the form~$f$ is the
generic cubic not tangent to $\iso$ that has the line $t$ in the eigenscheme.
If $d \neq 0$, the eigenpoints of~$f$ different from $P_1, P_2, P_3$ are
the common zeros of the following polynomials:
\begin{align*}
  h_1 & = 3dy^2 - 2axy - 2by^2 + bz^2 - 3cyz \,,\\
  h_2 & = 3dx^2 - 2ax^2 + az^2 - 2bxy - 3cxz \,.
\end{align*}
In general, the ideal $(h_1, h_2)$ gives $4$ points in
general position. \\
In case the line $t$ is tangent to the isotropic conic, we can
assume it is the line $x+\iii y =0$. Here we fix the three
points
\[
P_1 = (1: \iii: 0), \ P_2 = (0: 0: 1), \ P_3 = (1: \iii: 1) \,.
\]
In this case the linear system
$\Lambda \bigl(\Phi(P_1, P_2, P_3)\bigr)$ is four dimensional and
is given by:
\[
f = (x+\iii y)^2(ax + by+cz)+
 d(x^2 + y^2 + 2/3z^2)z+e (x^3 -\iii y^3 + z^3),
 \quad a, b, c, d, e \in \C
\]
If $e=0$, the form $f$ is described by \Cref{eq:cubics_with_tangent_eigenline}, thus it is the generic cubic which contains the line $t$ in the eigenscheme. If $e \not= 0$, as above,
the remaining eigenpoints of~$f$ are given by an ideal generated
by four polynomials $h_1, \dotsc, h_4$.
Also here, the zeros of general $h_1, \dotsc, h_4$ are in general position.
Notebook \nb{06}{F4} collects the computations for this proof.
\end{proof}

\subsection{Eigenschemes containing a conic}

\begin{theorem} Suppose that $\Gamma$ is a conic and let $C=V(f)$ be a cubic such that $\Gamma \subseteq \Eig{f}$. Then we have three possible cases:
    \begin{itemize}
        \item $\Gamma = \iso$. This is true if and only if $f = \ell\iso$ where $\ell$ is any line of the plane;
        \item $\Gamma$ is bitangent to $\iso$ in two distinct points $P_1$ and $P_2$. This is true iff $f = r(\lambda \iso+ \mu r^2)$, where $\lambda, \mu \in \mathbb{C}$, $r = P_1 \vee P_2$. In this case $\Gamma = V(\lambda \iso+3 \mu r^2)$;
        \item $\Gamma$ is iperosculating $\iso$ in a point $P$. This is true iff $f = r(\iso-r^2)$, where $r$ is the tangent to $\iso$ in $P$. In this case
        $\Gamma = V(\iso-3r^2)$.
    \end{itemize}
\end{theorem}
\begin{proof} Suppose $P_1, \dots, P_4$ are four distinct points on $\iso$ and let $C=V(f)$ be a cubic with $P_1, \dots, P_4\in \Eig{f}$. Then we claim that
$\iso\subseteq \Eig{f}$. Indeed, the linear system $\Lambda(\Phi(P_1, \dotsc, P_4))$ is two-dimensional and is contained in the linear system given by all the cubics of the form $\ell \iso$ (whose cubics have $\iso$ in the eigenscheme) and is also two dimensional, so they coincide. From this we have the first point. If $C$ is a cubic that contains in the eigenscheme a conic $\Gamma$ different from $\iso$, then $\Gamma$ and $\iso$ intersect in $4$ points. From the above result, the four points cannot be distinct, hence we have to consider four cases: $\Gamma$ is tangent to $\iso$ in one point, $\Gamma$ is bitangent to
$\iso$, $\Gamma$ is osculating $\iso$ and $\Gamma$ is iperosculating $\iso$. In all these cases we can assume that
$P_1 = (1: i: 0) \in \Gamma \cap \iso$ and that
the line $x+iy$ is the common tangent to $\Gamma$ and $\iso$ in $P_1$.
In the first case we consider the pencil of conics $\Gamma_s$ passing through $P_1, P_2, P_3\in \Gamma \cap \iso$ and tangent to $\iso$ in $P_1$ and
we take two distinct points $P_4, P_5$ on $\Gamma_s$. The matrix of conditions
$\Phi(P_1, \dotsc, P_5)$ must have rank $9$ or smaller, but the computations show that
this is not possible. In the second case, we take a generic point $P_2$ on $\iso$, we construct
the pencil of conics~$\Gamma_s$ which are bitangent to~$\iso$ in $P_1$ and $P_2$ and we take three other points $P_3, P_4, P_5$
on $\Gamma_s$. Again we check when the matrix of conditions $\Phi(P_1, \dotsc, P_5)$ has rank $9$ (or less). In this case we get that there is only one solution, which is given by the cubic
$V(r(\lambda \iso+\mu r^2))$. The remaining two cases are similar. The computational details can be find in \nb{06}{F5}.
\end{proof}
\begin{rmk}
    \Cref{lemma:twoTangentsCiso} can be seen as a particular case of the second item of the Theorem above.
\end{rmk}

\section{Possible configurations of the seven eigenpoints}
\label{further_alignments}

\begin{table}[ht]
\caption{All possible combinatorial configurations of seven points with at least one alignment and no six on a conic.}
\centering
\begin{tabular}{|clc|}\hline
  n. lines & collinear vertices & config.\\ \hline
 1& (1, 2, 3) &  $(C_1)$\\
 2& (1, 2, 3), (1, 4, 5) &  $(C_2)$\\
 3& (1, 2, 3), (1, 4, 5), (1, 6, 7) & $(C_3)$\\
 3& (1, 2, 3), (1, 4, 5), (2, 4, 6) & $(C_4)$\\
 4& (1, 2, 3), (1, 4, 5), (1, 6, 7), (2, 4, 6) & $(C_5)$\\
 4& (1, 2, 3), (1, 4, 5), (2, 4, 6), (3, 5, 6) & $(C_6)$\\
 5& (1, 2, 3), (1, 4, 5), (1, 6, 7), (2, 4, 6), (2, 5, 7)& $(C_7)$\\
 6& (1, 2, 3), (1, 4, 5), (1, 6, 7), (2, 4, 6), (2, 5, 7), (3, 4, 7)& $(C_8)$\\
 7& (1, 2, 3), (1, 4, 5), (1, 6, 7), (2, 4, 6), (2, 5, 7), (3, 4, 7), (3, 5, 6) &  $(C_9)$\\ \hline
\end{tabular}
\label{table:all_alignments}
\end{table}
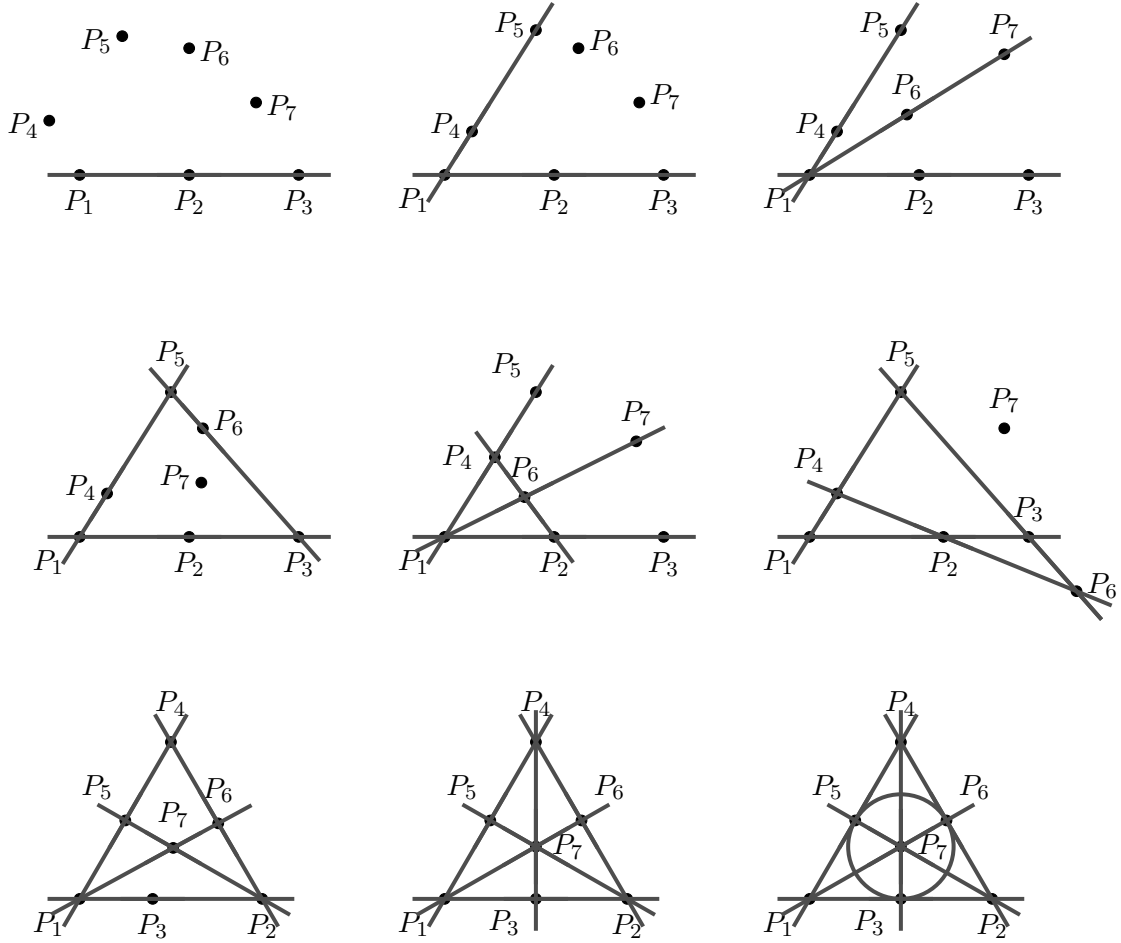
\begin{figure}[ht]
  \centering
  \begin{tikzpicture}
    \begin{scope}[scale=0.8]
    \begin{scope}
      \node[point, label={[label distance = -4pt, below]$P_1$}] (P1) at (0,0) {};
      \node[point, label={[label distance = -4pt, below]$P_3$}] (P3) at (3.6,0) {};
      \node[point, label={[label distance = -4pt, below]$P_2$}] (P2) at ($(P1)!0.5!(P3)$) {};
      \node[point, label={[label distance = -4pt, left]$P_4$}] (P4) at (-0.5, 0.9) {};
      \node[point, label={[label distance = -4pt, left]$P_5$}] (P5) at (0.7, 2.3) {};
      \node[point, label={[label distance = -4pt, right]$P_6$}] (P6) at (1.8, 2.1) {};
      \node[point, label={[label distance = -4pt, right]$P_7$}] (P7) at (2.9, 1.2) {};
      \draw[line, shorten <= -0.5cm, shorten >= -0.5cm] (P1) -- (P2);
      \draw[line, shorten <= -0.5cm, shorten >= -0.5cm] (P2) -- (P3);
    \end{scope}
    \begin{scope}[xshift=6cm]
      \node[point, label={[label distance = 0pt]210:$P_1$}] (P1) at (0,0) {};
      \node[point, label={[label distance = -4pt, below]$P_3$}] (P3) at (3.6,0) {};
      \node[point, label={[label distance = -4pt, below]$P_2$}] (P2) at ($(P1)!0.5!(P3)$) {};
      \node[point, label={[label distance = 0pt, left]$P_5$}] (P5) at (1.5, 2.4) {};
      \node[point, label={[label distance = 0pt, left]$P_4$}] (P4) at ($(P1)!0.3!(P5)$) {};
      \node[point, label={[label distance = 0pt, right]$P_6$}] (P6) at (2.2, 2.1) {};
      \node[point, label={[label distance = 0pt, right]$P_7$}] (P7) at (3.2, 1.2) {};
      \draw[line, shorten <= -0.5cm, shorten >= -0.5cm] (P1) -- (P2);
      \draw[line, shorten <= -0.5cm, shorten >= -0.5cm] (P2) -- (P3);
      \draw[line, shorten <= -0.5cm, shorten >= -0.5cm] (P1) -- (P4);
      \draw[line, shorten <= -0.5cm, shorten >= -0.5cm] (P4) -- (P5);
    \end{scope}
    \begin{scope}[xshift=12cm]
      \node[point, label={[label distance = 0pt]210:$P_1$}] (P1) at (0,0) {};
      \node[point, label={[label distance = -4pt, below]$P_3$}] (P3) at (3.6,0) {};
      \node[point, label={[label distance = -4pt, below]$P_2$}] (P2) at ($(P1)!0.5!(P3)$) {};
      \node[point, label={[label distance = 0pt, left]$P_5$}] (P5) at (1.5, 2.4) {};
      \node[point, label={[label distance = 0pt, left]$P_4$}] (P4) at ($(P1)!0.3!(P5)$) {};
      \node[point, label={[label distance = 0pt, above]$P_6$}] (P6) at (1.6, 1) {};
      \node[point, label={[label distance = 0pt, above]$P_7$}] (P7) at (3.2, 2) {};
      \draw[line, shorten <= -0.5cm, shorten >= -0.5cm] (P1) -- (P2);
      \draw[line, shorten <= -0.5cm, shorten >= -0.5cm] (P2) -- (P3);
      \draw[line, shorten <= -0.5cm, shorten >= -0.5cm] (P1) -- (P4);
      \draw[line, shorten <= -0.5cm, shorten >= -0.5cm] (P4) -- (P5);
      \draw[line, shorten <= -0.5cm, shorten >= -0.5cm] (P1) -- (P6);
      \draw[line, shorten <= -0.5cm, shorten >= -0.5cm] (P6) -- (P7);
    \end{scope}
    \begin{scope}[yshift=-6cm]
      \node[point, label={[label distance = 0pt]210:$P_1$}] (P1) at (0,0) {};
      \node[point, label={[label distance = -4pt, below]$P_3$}] (P3) at (3.6,0) {};
      \node[point, label={[label distance = -4pt, below]$P_2$}] (P2) at ($(P1)!0.5!(P3)$) {};
      \node[point, label={[label distance = 4pt]90:$P_5$}] (P5) at (1.5, 2.4) {};
      \node[point, label={[label distance = 0pt, left]$P_4$}] (P4) at ($(P1)!0.3!(P5)$) {};
      \node[point, label={[label distance = 0pt, right]$P_6$}] (P6) at ($(P5)!0.25!(P3)$) {};
      \node[point, label={[label distance = 0pt, left]$P_7$}] (P7) at (2, 0.9) {};
      \draw[line, shorten <= -0.5cm, shorten >= -0.5cm] (P1) -- (P2);
      \draw[line, shorten <= -0.5cm, shorten >= -0.5cm] (P2) -- (P3);
      \draw[line, shorten <= -0.5cm, shorten >= -0.5cm] (P1) -- (P4);
      \draw[line, shorten <= -0.5cm, shorten >= -0.5cm] (P4) -- (P5);
      \draw[line, shorten <= -0.5cm, shorten >= -0.5cm] (P3) -- (P6);
      \draw[line, shorten <= -0.5cm, shorten >= -0.5cm] (P6) -- (P5);
    \end{scope}
    \begin{scope}[xshift=6cm, yshift=-6cm]
      \node[point, label={[label distance = 0pt]210:$P_1$}] (P1) at (0,0) {};
      \node[point, label={[label distance = -4pt, below]$P_3$}] (P3) at (3.6,0) {};
      \node[point, label={[label distance = -4pt, below]$P_2$}] (P2) at ($(P1)!0.5!(P3)$) {};
      \node[point, label={[label distance = 0pt]120:$P_5$}] (P5) at (1.5, 2.4) {};
      \node[point, label={[label distance = 2pt]180:$P_4$}] (P4) at ($(P1)!0.55!(P5)$) {};
      \node[point, label={[label distance = 0pt]90:$P_6$}] (P6) at ($(P2)!0.5!(P4)$) {};
      \node[point, label={[label distance = 0pt, above]$P_7$}] (P7) at ($(P1)!2.4!(P6)$) {};
      \draw[line, shorten <= -0.5cm, shorten >= -0.5cm] (P1) -- (P2);
      \draw[line, shorten <= -0.5cm, shorten >= -0.5cm] (P2) -- (P3);
      \draw[line, shorten <= -0.5cm, shorten >= -0.5cm] (P1) -- (P4);
      \draw[line, shorten <= -0.5cm, shorten >= -0.5cm] (P4) -- (P5);
      \draw[line, shorten <= -0.5cm, shorten >= -0.5cm] (P1) -- (P6);
      \draw[line, shorten <= -0.5cm, shorten >= -0.5cm] (P6) -- (P7);
      \draw[line, shorten <= -0.5cm, shorten >= -0.5cm] (P2) -- (P6);
      \draw[line, shorten <= -0.5cm, shorten >= -0.5cm] (P6) -- (P4);
    \end{scope}
    \begin{scope}[xshift=12cm, yshift=-6cm]
      \node[point, label={[label distance = 0pt]210:$P_1$}] (P1) at (0,0) {};
      \node[point, label={[label distance = 0pt, above]$P_3$}] (P3) at (3.6,0) {};
      \node[point, label={[label distance = -4pt, below]$P_2$}] (P2) at (2.2,0) {};
      \node[point, label={[label distance = 4pt]90:$P_5$}] (P5) at (1.5, 2.4) {};
      \node[point, label={[label distance = 2pt]100:$P_4$}] (P4) at ($(P1)!0.3!(P5)$) {};
      \tkzInterLL(P3,P5)(P2,P4) \tkzGetPoint{P6}
      \node[point, label={[label distance = 0pt, right]$P_6$}] at (P6) {};
      \node[point, label={[label distance = 0pt, above]$P_7$}] (P7) at (3.2, 1.8) {};
      \draw[line, shorten <= -0.5cm, shorten >= -0.5cm] (P1) -- (P2);
      \draw[line, shorten <= -0.5cm, shorten >= -0.5cm] (P2) -- (P3);
      \draw[line, shorten <= -0.5cm, shorten >= -0.5cm] (P1) -- (P4);
      \draw[line, shorten <= -0.5cm, shorten >= -0.5cm] (P4) -- (P5);
      \draw[line, shorten <= -0.5cm, shorten >= -0.5cm] (P3) -- (P6);
      \draw[line, shorten <= -0.5cm, shorten >= -0.5cm] (P6) -- (P5);
      \draw[line, shorten <= -0.5cm, shorten >= -0.5cm] (P4) -- (P2);
      \draw[line, shorten <= -0.5cm, shorten >= -0.5cm] (P2) -- (P6);
    \end{scope}
    \begin{scope}[yshift=-12cm]
      \node[point, label={[label distance = 0pt]210:$P_1$}] (P1) at (0,0) {};
      \node[point, label={[label distance = -4pt, below]$P_2$}] (P2) at (3,0) {};
      \node[point, label={[label distance = -4pt, below]$P_3$}] (P3) at ($(P1)!0.4!(P2)$) {};
      \node[point, label={[label distance = 4pt]90:$P_4$}] (P4) at (1.5, 2.6) {};
      \node[point, label={[label distance = 2pt]100:$P_5$}] (P5) at ($(P1)!0.5!(P4)$) {};
      \node[point, label={[label distance = 2pt, above]:$P_7$}] (P7) at ($(P2)!0.65!(P5)$) {};
      \tkzInterLL(P1,P7)(P2,P4) \tkzGetPoint{P6}
      \node[point, label={[label distance = 2pt, above]$P_6$}] at (P6) {};
      \draw[line, shorten <= -0.5cm, shorten >= -0.5cm] (P1) -- (P3);
      \draw[line, shorten <= -0.5cm, shorten >= -0.5cm] (P3) -- (P2);
      \draw[line, shorten <= -0.5cm, shorten >= -0.5cm] (P1) -- (P5);
      \draw[line, shorten <= -0.5cm, shorten >= -0.5cm] (P5) -- (P4);
      \draw[line, shorten <= -0.5cm, shorten >= -0.5cm] (P2) -- (P6);
      \draw[line, shorten <= -0.5cm, shorten >= -0.5cm] (P6) -- (P4);
      \draw[line, shorten <= -0.5cm, shorten >= -0.5cm] (P5) -- (P7);
      \draw[line, shorten <= -0.5cm, shorten >= -0.5cm] (P7) -- (P2);
      \draw[line, shorten <= -0.5cm, shorten >= -0.5cm] (P1) -- (P7);
      \draw[line, shorten <= -0.5cm, shorten >= -0.5cm] (P7) -- (P6);
    \end{scope}
    \begin{scope}[xshift=6cm, yshift=-12cm]
      \node[point, label={[label distance = 0pt]210:$P_1$}] (P1) at (0,0) {};
      \node[point, label={[label distance = -4pt, below]$P_{2}$}] (P2) at (3,0) {};
      \node[point, label={[label distance = 0pt]195:$P_{3}$}] (P3) at ($(P1)!0.5!(P2)$) {};
      \node[point, label={[label distance = 4pt]90:$P_{4}$}] (P4) at (1.5, 2.6) {};
      \node[point, label={[label distance = 2pt]100:$P_{5}$}] (P5) at ($(P1)!0.5!(P4)$) {};
      \node[point, label={[label distance = 2pt]80:$P_6$}] (P6) at ($(P2)!0.5!(P4)$) {};
      \tkzInterLL(P2,P5)(P3,P4) \tkzGetPoint{P7}
      \node[point, label={[label distance = 0pt]0:$P_7$}] at (P7) {};
      \draw[line, shorten <= -0.5cm, shorten >= -0.5cm] (P1) -- (P3);
      \draw[line, shorten <= -0.5cm, shorten >= -0.5cm] (P3) -- (P2);
      \draw[line, shorten <= -0.5cm, shorten >= -0.5cm] (P1) -- (P5);
      \draw[line, shorten <= -0.5cm, shorten >= -0.5cm] (P5) -- (P4);
      \draw[line, shorten <= -0.5cm, shorten >= -0.5cm] (P2) -- (P6);
      \draw[line, shorten <= -0.5cm, shorten >= -0.5cm] (P6) -- (P4);
      \draw[line, shorten <= -0.5cm, shorten >= -0.5cm] (P5) -- (P7);
      \draw[line, shorten <= -0.5cm, shorten >= -0.5cm] (P7) -- (P2);
      \draw[line, shorten <= -0.5cm, shorten >= -0.5cm] (P4) -- (P7);
      \draw[line, shorten <= -0.5cm, shorten >= -0.5cm] (P7) -- (P3);
      \draw[line, shorten <= -0.5cm, shorten >= -0.5cm] (P1) -- (P7);
      \draw[line, shorten <= -0.5cm, shorten >= -0.5cm] (P7) -- (P6);
    \end{scope}
    \begin{scope}[xshift=12cm, yshift=-12cm]
      \node[point, label={[label distance = 0pt]210:$P_1$}] (P1) at (0,0) {};
      \node[point, label={[label distance = -4pt, below]$P_{2}$}] (P2) at (3,0) {};
      \node[point, label={[label distance = 0pt]195:$P_{3}$}] (P3) at ($(P1)!0.5!(P2)$) {};
      \node[point, label={[label distance = 4pt]90:$P_{4}$}] (P4) at (1.5, 2.6) {};
      \node[point, label={[label distance = 2pt]100:$P_{5}$}] (P5) at ($(P1)!0.5!(P4)$) {};
      \node[point, label={[label distance = 2pt]80:$P_6$}] (P6) at ($(P2)!0.5!(P4)$) {};
      \tkzInterLL(P2,P5)(P3,P4) \tkzGetPoint{P7}
      \node[point, label={[label distance = 0pt]0:$P_7$}] at (P7) {};
      \node[draw, line] at (P7) [circle through={(P3)}] {};
      \draw[line, shorten <= -0.5cm, shorten >= -0.5cm] (P1) -- (P3);
      \draw[line, shorten <= -0.5cm, shorten >= -0.5cm] (P3) -- (P2);
      \draw[line, shorten <= -0.5cm, shorten >= -0.5cm] (P1) -- (P5);
      \draw[line, shorten <= -0.5cm, shorten >= -0.5cm] (P5) -- (P4);
      \draw[line, shorten <= -0.5cm, shorten >= -0.5cm] (P2) -- (P6);
      \draw[line, shorten <= -0.5cm, shorten >= -0.5cm] (P6) -- (P4);
      \draw[line, shorten <= -0.5cm, shorten >= -0.5cm] (P5) -- (P7);
      \draw[line, shorten <= -0.5cm, shorten >= -0.5cm] (P7) -- (P2);
      \draw[line, shorten <= -0.5cm, shorten >= -0.5cm] (P4) -- (P7);
      \draw[line, shorten <= -0.5cm, shorten >= -0.5cm] (P7) -- (P3);
      \draw[line, shorten <= -0.5cm, shorten >= -0.5cm] (P1) -- (P7);
      \draw[line, shorten <= -0.5cm, shorten >= -0.5cm] (P7) -- (P6);
    \end{scope}
    \end{scope}
  \end{tikzpicture}
  \caption{Graphical representations of the $9$ cases of possible alignments of seven points as described in \Cref{table:all_alignments}.}
  \label{figure:all_alignments}
\end{figure}
In this section, we want to identify which of the nine configurations in \Cref{table:all_alignments} can be realized by
the seven eigenpoints of a regular ternary cubic polynomial and, if so, which are the cubics with that configuration of eigenpoints.

We say that the eigenpoints of a cubic curve are in a \emph{$(C_i)$
configuration} if they are aligned according to~$(C_i)$; we say
that the eigenpoints are in a \emph{strict $(C_i)$ configuration} if,
in addition, there are no further alignments among them.

First of all, it is well known that configuration~$(C_9)$ cannot be realized
by seven points of the plane over a field of zero
characteristic (see \cite{Whitney1935}), therefore we do not consider
it in our analysis.

\subsection*{Configuration~\texorpdfstring{$(C_1)$}{C1}}
This configuration can be realized. \Cref{proposition:three_distinct_ranks} and \Cref{locus_one_alignment}
give a description of the cubics with such a configuration of points:
we fix two points~$P_1$ and~$P_2$ in~$\p^2$, we take a generic~$P_3$
on the line~$P_1 \vee P_2$, all the cubics with
$P_1, P_2, P_3$ eigenpoints are given by $\Lambda \bigl( \Phi(P_1, P_2, P_3) \bigr)$, the three
dimensional linear subspace of~$\p^9$ (four dimensional, if the
line~$P_1 \vee P_2$ is tangent in~$P_1$, $P_2$, or~$P_3$ to~$\iso$). For generic choices of the points, the cubics have the eigenpoints in a strict
$(C_1)$ configuration.

\subsection*{Configuration~\texorpdfstring{$(C_2)$}{C2}}
In this case, the points
$P_1, \dots, P_5$ are in a
$V$- configuration, so the rank of~$\Phi(P_1, \dotsc, P_5)$
must be~$9$ or~$8$. If the rank is~$9$, then $\delta_1(P_1, P_2, P_4) = 0$
or $\delta_2(P_1, \dotsc, P_5) = 0$, see \Cref{theorem:rank_V}.
From \Cref{proposition:third_alignment}, the only case
which admits a strict $(C_2)$ configuration
is $\delta_1(P_1, P_2, P_4) = 0$. Therefore, if we fix two points~$P_1$ and~$P_2$
in the plane in an arbitrary way, from \Cref{lemma_delta_case1} we can choose $P_4$ so that
$\scl{P_4}{s_{11}\, P_2 - s_{12} \, P_1}=0$, then we fix any~$P_3$
on the line~$P_1 \vee P_2$ and any~$P_5$ on the line~$P_1 \vee P_4$; in this way, we get a cubic with a configuration of type~$(C_2)$, which is generally strict. If the rank is~$8$, configuration~$(C_2)$ can only be obtained from \Cref{rk8_1}, hence we get sub-cases of the case $\delta_1(P_1, P_2, P_4)=0$, see~\Cref{remark:particular_cases}.

\subsection*{Configuration~\texorpdfstring{$(C_3)$}{C3}}
If we have the alignments $(P_1, P_2, P_3)$, $(P_1, P_4, P_5)$ and $(P_1, P_6, P_7)$, then
\[
 (P_1, P_2, P_3, P_4, P_5) \,, \quad (P_1, P_2, P_3, P_6, P_7) \,, \quad (P_1, P_4, P_5, P_6, P_7)
\]
are three $V$- configurations. It holds:

\begin{lemma}
\label{lemma:no_delta1_delta1}
Suppose we have seven eigenpoints $P_1, \dots, P_7$
of a cubic in configuration $(C_3)$. Then among the $7$ points there is a
$V$- configuration that satisfies a $\delta_2 = 0$ condition.
\end{lemma}
\begin{proof}
The points $P_1, P_2, P_3, P_4, P_5$ are in a $V$- configuration.
If $\rk \, \Phi(P_1, \dots, P_5) = 8$, the result follows from~\Cref{rank_8}.
Therefore, assume that the matrix $\Phi(P_1, \dots, P_5)$ has rank~$9$.
If $\delta_2(P_1, \dots, P_5) = 0$, the statement holds;
otherwise, $\delta_1(P_1, P_2, P_4) = 0$.
Then consider the $V$- configuration $P_1$, $P_2$, $P_3$, $P_6$, $P_7$.
As above, we can suppose $\delta_1(P_1, P_2, P_6) = 0$.
These two equations are
linear in the coordinates of~$P_2$.
If the matrix of the associated linear system has
maximal rank, the unique solution gives a point~$P_2$ which coincides, as a projective point, to~$P_1$, which is impossible.
Thus the matrix does not have maximal rank.
This condition implies that $P_1$ is on the
isotropic conic. As usual, we can assume
$P_1 = (1: \iii: 0)$ and again
we can determine~$P_2$ so that $\delta_1(P_1, P_2, P_4)$ and
$\delta_1(P_1, P_2, P_6)$ are zero. The matrix $\Phi(P_1, P_4, P_5, P_6, P_7)$
must have rank~$9$ or smaller; then either
$\delta_2(P_1, P_4, P_5, P_6, P_7)=0$ or $\delta_1(P_1, P_4, P_6) = 0$. In
the first case, we have a $\delta_2$ condition among the points, hence
we assume $\delta_1(P_1, P_4, P_6) = 0$. Analyzing this equation, we get that either
$\delta_2(P_1, P_2, P_3, P_6, P_7) = 0$ or
$\delta_2(P_1, P_2, P_3, P_4, P_5) = 0$.
See  \nb{07}{F1} for the details.
\end{proof}

As a consequence of \Cref{lemma:no_delta1_delta1}, configuration $(C_3)$ can
be obtained only if we have $5$ points such that
$\delta_2(P_1, \dotsc, P_5) = 0$. \Cref{proposition:definitionP3} describes how to
obtain five eigenpoints in a $(C_3)$ configuration of eigenpoints; observe that when $P_3$ is defined by \Cref{proposition:definitionP3}, Item~(4), the configuration is strict.

\subsection*{Configuration~\texorpdfstring{$(C_4)$}{C4}}
It holds:
\begin{prop}
If the eigenpoints of a cubic are in a $(C_4)$ configuration, then they actually are in a $(C_8)$ configuration.
\end{prop}
\begin{proof}
From the results of \Cref{rank_8}, we know that $(C_4)$
cannot be obtained from a rank $8$ $V$- configuration,
hence it remains to consider the case $\delta_1(P_1, P_2, P_4) = 0$,
$\delta_1(P_2, P_1, P_4) = 0$ and $\delta_1(P_4, P_1, P_2) = 0$,
which implies
\[
  \delta_2(P_1, P_2, P_3, P_4, P_5) = 0 \,, \quad
  \delta_2(P_2, P_1, P_3, P_4, P_6) = 0 \,, \quad
  \delta_2(P_4, P_1, P_5, P_2, P_6) = 0 \,,
\]
which gives rise to $(C_8)$ configuration. The computations are available in \nb{07}{F2}.
\end{proof}

\subsection*{Configuration~\texorpdfstring{$(C_5)$}{C5}}
\begin{prop}
\label{proposition:condition_5}
Suppose we have $7$ points which are eigenpoints of a cubic and are in a strict~$(C_5)$ configuration.
Then it holds
\[
  P_1 = P_2 \times P_4
\]
and among the points $P_1, \dots, P_6$ we have the relation
\begin{equation}
\label{eq:condition_C5}
  s_{26}(s_{45}s_{13}-s_{34}s_{15})+s_{46}(s_{25}s_{13}-s_{23}s_{15}) = 0 \,.
\end{equation}
Moreover, the points~$P_6$ and~$P_7$ can be determined as follows:
\begin{equation}
\label{eq:p6formula}
\begin{multlined}
  P_6 = (s_{15}s_{24}s_{34}+s_{15}s_{23}s_{44} -s_{13}s_{25}s_{44} -s_{13}s_{24}s_{45}) \, P_2 \\ + (s_{13}s_{24}s_{25}-2s_{15}s_{22}s_{34}+s_{13}s_{22}s_{45}) \, P_4 \,,
\end{multlined}
\end{equation}
\begin{equation}
\label{eq:p7formula}
P_7 = (s_{26}s_{15}s_{46}+s_{24}s_{15}s_{66})P_1 + s_{11}(s_{26}s_{45}+s_{24}s_{56})P_6\,.
\end{equation}
\end{prop}
\begin{proof}
The points $(P_2, P_1, P_3, P_4, P_6)$ are in a
$V$- configuration. If it is of rank $8$,
from \Cref{rank_8} the seven points are in a $(C_3)$ configuration
with a collinearity $P_2, P_3,P_5$; by assumption, this is not the case. So
we can assume that $\rk \, \Phi(P_1, P_2, P_3, P_4, P_6) = 9$, hence $\delta_1(P_2, P_1, P_4)=0$.
Arguing similarly, we get $\delta_1(P_4, P_1, P_2) = 0$ and $\delta_1(P_6, P_1, P_2) = 0$. Moreover, by \Cref{lemma:no_delta1_delta1} and by possibly relabelling the points, we may assume that $\delta_2(P_1, P_2, P_3, P_4, P_5)=0$. From this we get that
$s_{12} = s_{14}=s_{16}=0$. In particular $P_1 = P_2 \times P_4$.

To prove \Cref{eq:condition_C5},
we consider the matrix $M = \Phi(P_1, \dots, P_6)$;
for $P_6$ to be an eigenpoint,
all the order $10$-minors must be zero.
After suitable saturations, the ideal of such minors is principal and the generator gives \Cref{eq:condition_C5}.

If in \Cref{eq:condition_C5} in place of~$P_6$ we substitute $w_1 \, P_2 + w_2 \, P_4$,
we find~$w_1$ and~$w_2$, which give \Cref{eq:p6formula}.

Finally, to prove \Cref{eq:p7formula}, we change~$P_3$ with~$P_7$ in \Cref{eq:condition_C5} and we take into account that $P_7$ is collinear with~$P_1$ and~$P_6$. All the detailed computations can be found in \nb{07}{F3}.
\end{proof}
The converse of the above proposition is also true, more precisely:
\begin{prop}
   Suppose $P_1, P_2, P_3, P_4, P_5$ are in a $V$- configuration such that
\[
  P_1 = P_2 \times P_4 \,,
\]
and suppose there is a cubic that has $7$ eigenpoints, five of which are $P_1, \dotsc, P_5$. Then the remaining eigenpoints~$P_6$ and~$P_7$ are given by \Cref{eq:p6formula} and \Cref{eq:p7formula}. Therefore $(P_2, P_4, P_6)$ are aligned, so the points are in a $(C_5)$ configuration.

Furthermore, in the general case, the configuration is
strict, but
there are sub-cases in which $(P_2, P_5, P_7)$ or $(P_3, P_4, P_7)$ or $(P_3, P_5, P_7)$ are aligned. In all these cases the points are in a $(C_8)$ configuration.
\end{prop}
\begin{proof}
The condition $P_1 = P_2 \times P_4$ implies $s_{12}=s_{14}=0$,
so we have $\delta_2 (P_1,P_2,P_3,P_4,P_5)=0$.
First, we suppose that $\rk \, \Phi(P_1, \dotsc, P_5) = 9$.
It follows from \Cref{proposition:third_alignment} that the corresponding unique cubic having such $5$ points as eigenpoints has $P_6$ and~$P_7$ aligned with~$P_1$.
To prove that $P_6$ and $P_7$ are given by \Cref{eq:p6formula} and \Cref{eq:p7formula},
we check that the matrix of conditions of $P_1,\cdots, P_5$ and the point of
\eqref{eq:p6formula}, respectively of \eqref{eq:p7formula}, has rank $9$.

A random example shows that, in general, there are no further collinearities.
The condition that $(P_2, P_5, P_7)$ are aligned splits into two cases:
one forces the collinearity of $(P_3, P_4, P_7)$, the other imposes the collinearity of $(P_3, P_5, P_6)$,
hence we get two $(C_8)$ configurations. Similarly for the other two cases.

If $\rk \, \Phi(P_1, \dotsc, P_5) = 8$, the $V$- configuration cannot satisfy \Cref{rk8_1}. Indeed, we would have \Cref{lemma_delta_case2} and by $s_{12} = 0$, it would be $P_1 = P_3$. Thus,
from the results of \Cref{rank_8}, the lines $P_1 \vee P_2$ and $P_1 \vee P_4$ are tangent to the isotropic conic in~$P_2$ and~$P_4$. Thus
\[
s_{22}= s_{44}= s_{23}= s_{45}=0
\]
and \Cref{eq:p6formula} and \Cref{eq:p7formula} specialize to \Cref{eq:formulaeP6_P7}. Finally, \Cref{proposition:rk8_2B} shows that the points are in a $(C_5)$ configuration. Notebook \nb{07}{F3} contains the details of the computations.
\end{proof}

\begin{rmk}
\label{remark:C5rk8}
Since in a $(C_5)$ configuration we have
$P_1 = P_2 \times P_4$ and $(P_2, P_4, P_6)$ aligned, it follows that $s_{12} = 0$, $s_{14} = 0$ and $s_{16} = 0$.
Moreover, it is
immediate to verify that the line~$P_2 \vee P_4$ is orthogonal to the
lines~$P_1 \vee P_2$, $P_1 \vee P_4$, and~$P_1 \vee P_6$, that is
\[
  \scl{P_1 \times P_2}{P_2 \times P_4} = 0, \quad
  \scl{P_1 \times P_4}{P_2 \times P_4} = 0, \quad
  \scl{P_1 \times P_6}{P_2 \times P_4} = 0.
\]
Hence \Cref{remark:three_orthog} gives that the $(C_5)$ configuration obtained in \Cref{proposition:rk8_2B} is a sub-case of the $(C_5)$ configuration described here, since Equations \eqref{eq:p6formula} and \eqref{eq:p7formula} specialize to \eqref{eq:formulaeP6_P7}.
\end{rmk}

\subsection*{Configurations~\texorpdfstring{$(C_6)$}{C6} and \texorpdfstring{$(C_7)$}{C7}}

The strict configurations $(C_6)$ and~$(C_7)$ cannot be realized by eigenpoints.
Indeed, if we study the ideals given by the conditions
$\delta_1=0$ and~$\delta_2=0$ that must be satisfied,
we see that there are no compatible solutions.
See Notebooks \nb{07}{F4} and \nb{07}{F5}.

\subsection*{Configuration~\texorpdfstring{$(C_8)$}{C8}}
This configuration is realizable.
Observe that the points $P_1, P_2, P_4, P_7$ lie on three lines, while the points
$P_3, P_5, P_6$ lie on two lines.
The following lemma is easy to verify:

\begin{lemma}
\label{lemma:6ortog}
If we have four distinct points of the projective
plane, such that no triplets of them are collinear, then the scalar
product of at least two of them is not zero.
\end{lemma}

Furthermore:

\begin{lemma}
\label{lemma:three_s_zero}
Suppose that $P_1, \dotsc, P_7$ are eigenpoints in a $(C_8)$ configuration.
We also assume that the matrix of conditions of any $V$- configuration contained in $P_1, \dotsc, P_7$ has rank~$9$.
Let $U, V, W$ be three points in the set $\{P_1, P_2, P_4, P_7\}$
such that $\scl{U}{V} = 0$ and $\scl{U}{W} = 0$.
Then $\scl{V}{W} = 0$.
\end{lemma}
\begin{proof}
From the symmetries of the points, we can assume that $U=P_1$, $V=P_2$, $W=P_7$.
Then we have $P_1 = P_2 \times P_7$.
The matrix of conditions of the $V$- configuration
$(P_3, P_1, P_2, P_4, P_7)$
is of rank~$9$, hence $\delta_1(P_3, P_1, P_4) = 0$
(see~\Cref{theorem:rank_V} and~\Cref{proposition:third_alignment}).
Similarly, $\delta_1(P_5, P_1, P_2) = 0$ and $\delta_1(P_6, P_1, P_2) = 0$.
The ideal generated by such
$\delta_1$ conditions turns out,
after some saturations, to be generated by $s_{27}$.
More details on the file \nb{07}{F6}.
\end{proof}

\begin{prop}
\label{proposition:conf8_partA}
If $P_1, \dotsc, P_7$ are eigenpoints in a $(C_8)$ configuration, then~$s_{12}$ is not zero and it holds:
\begin{equation}
\label{eq:orthocenter}
P_7 = (P_1 \times P_2)s_{14}s_{24} -
  s_{12}(P_1 \times P_4)s_{24} + s_{12}s_{14}(P_2 \times P_4)
\end{equation}
and the lines $P_1 \vee P_2$, $P_1 \vee P_4$ and~$P_1 \vee P_6$
are orthogonal to, respectively, $P_3 \vee P_4$, $P_2 \vee P_5$ and
$P_2 \vee P_4$, up to a permutation of
the points. In particular, we have
\[
  \scl{P_1 \times P_2}{P_3 \times P_4} = 0, \quad
  \scl{P_1 \times P_4}{P_2 \times P_5} = 0, \quad
  \scl{P_1 \times P_6}{P_2 \times P_4} = 0 \,.
\]
\end{prop}
\begin{proof}
If some of the $V$- configurations contained in $(C_8)$ have matrices of conditions of rank~$8$,
the result follows from \Cref{rank_8} (see \Cref{remark:three_orthog}).

Assume now that each $V$- configurations of $(C_8)$ gives rise to a rank $9$ matrix of conditions.
As a consequence of the previous two lemmas, it is not possible
to have $s_{12}=s_{14}=s_{17}=0$, we can assume $s_{12} \not=0$.
Again from \Cref{lemma:three_s_zero}, it is not possible to
have $s_{14}=s_{24}=0$, so at least one of the three coefficients~$s_{14}s_{24}$, $s_{12}s_{24}$, or~$s_{12}s_{14}$ in \Cref{eq:orthocenter} is not zero. Since the
three vectors $P_1\times P_2$, $P_1\times P_4$, $P_2\times P_4$ are linearly independent, we see that $P_7$ of \eqref{eq:orthocenter} is well defined.

On the other hand, the eigenpoint conditions imply
$\delta_1(P_3, P_1, P_4) = 0$, $\delta_1(P_5, P_1, P_2) = 0$ and
$\delta_1(P_6, P_1, P_2)=0$. These three relations, up to non zero factors, can be expressed as
\begin{align*}
 e_1 &= s_{12}s_{47}-s_{17}s_{24} =0\,,\\
 e_2 &= s_{12}s_{47}-s_{14}s_{27} =0\,,\\
 e_3 &= s_{17}s_{24}-s_{14}s_{27} =0\,,
\end{align*}
hence $e_1-e_2+e_3 = 0$ and the system
$e_1=0, e_2 = 0$ is linear in the coordinates of the $P_7$.
Its solution determines~$P_7$ and it is easy to check that \Cref{eq:orthocenter} holds. The complete computations are in \nb{07}{F7}.
\end{proof}

We have also the converse, i.e.\
\begin{prop}
Suppose $P_1, P_2, P_4$ are three points, that \Cref{eq:orthocenter} defines a point $P_7$ and that
\[
P_3 = (P_1 \vee P_2) \cap(P_4 \vee P_7), \quad
P_5 = (P_1 \vee P_4) \cap (P_2 \vee P_7), \quad
P_6 = (P_1 \vee P_7) \cap (P_2 \vee P_4).
\]
Then the points $P_1, \dotsc, P_7$ are in a $(C_8)$ configuration
and are eigenpoints of a unique cubic
of the plane. In particular:
\[
  \scl{P_1 \times P_2}{P_3 \times P_4} = 0 \,, \quad
  \scl{P_1 \times P_4}{P_2 \times P_5} = 0 \,, \quad
  \scl{P_1 \times P_6}{P_2 \times P_4} = 0 \,.
\]
\end{prop}
\begin{proof}
It is enough to show that the rank of the matrix $\Phi(P_1, \dots, P_7)$
is $9$ and for this computation we can assume that
$P_1$ is the point~$(1: 0: 0)$ or the point~$(1: \iii: 0)$;
see \nb{07}{F7}.
\end{proof}

\begin{corollary}
The locus of ternary cubic forms with a $(C_8)$ configuration has dimension~$6$.
\end{corollary}

\subsection*{Comparison with ODECO tensors}
As mentioned in the introduction, a class of
symmetric tensors fitting in the framework of collinearities in the eigenscheme is represented by ODECO tensors, which were introduced by~\cite{Rob} and studied by \cite{BDHE, Koiran2021, Biaggi2022}.

Possibly after an $\SO_3(\C)$ transformation, such forms are of the type
\[
  f = \lambda_1 x^3 +\lambda_2 y^3 + \lambda_3 z^3
  \quad \text{for some } \lambda_1, \lambda_2, \lambda_3 \in \C \,,
\]
and, using the language of \cite{Rob}, they admit the following three pairwise orthogonal eigenvectors $Q_1, Q_2, Q_3 \in \C^3$
\[
  Q_1 = \left( \frac{1}{\lambda_1},0,0 \right) \,, \quad
  Q_2 = \left( 0,\frac{1}{\lambda_2},0 \right) \,, \quad
  Q_3 = \left( 0,0,\frac{1}{\lambda_3} \right) \,.
\]
The remaining $4$ eigenvectors are uniquely determined, precisely:
\[
  Q_4 = Q_1+Q_2+Q_3\,, \ Q_5 = Q_1+Q_2,\, \ Q_6 = Q_1+Q_3 \,, \ Q_7 = Q_2+Q_3\,.
\]
In particular, such eigenschemes are in a $(C_8)$ configuration, with the additional condition
\[
  \left\langle Q_i,Q_j \right\rangle = 0
  \quad \text{ for } (i, j) \in \{(1, 2), (1, 3), (2, 3), (1, 7), (2, 6), (3, 5)\} \,.
\]
The seven vectors fit into the description of \Cref{proposition:conf8_partA}, since $Q_3 $ is proportional to $Q_1 \times Q_2$ and this is coherent with \Cref{eq:orthocenter}
(because $\scl{Q_1}{Q_2} = 0$)
and
\[
  (Q_1 \vee Q_2) \times (Q_3 \vee Q_5) = 0 \,, \quad
  (Q_1 \vee Q_3) \times (Q_2 \vee Q_6) = 0 \,, \quad
  (Q_2 \vee Q_3) \times (Q_1 \vee Q_7) = 0 \,.
\]
In particular, the locus of ODECO forms is a $5$-dimensional sublocus of the $(C_8)$ locus.

\medskip
In Notebook \nb{07}{F8} it is possible to find a construction of a
generic cubic which has a $(C_i)$ configuration for all possible $i$.

\subsection*{Final remark} It is a challenging question to identify the irreducible components and the degrees of the loci of $(C_i)$ configurations.

\bibliographystyle{alphaurl}
\bibliography{biblio}

\end{document}